\documentclass[oneside]{scrartcl}

\setkomafont{sectioning}{\bfseries}
\usepackage[utf8]{inputenc}\usepackage{lmodern}
\usepackage[T1]{fontenc}
\usepackage{verbatim}
\usepackage{cancel}\usepackage{enumitem}
\usepackage{amssymb, color}\usepackage[sans]{dsfont}
\usepackage{amsbsy,bbm,fontawesome,nicefrac,hyperref,bclogo,multirow,yhmath,bbding}
\usepackage{soul}

\usepackage{amsmath,amssymb,amsthm,mathrsfs,cancel}
\usepackage{natbib}

\makeatletter
\theoremstyle{plain}

\newtheorem{theorem}{Theorem}
\newtheorem{proposition}[theorem]{Proposition}                 
\newtheorem{lemma}[theorem]{Lemma}

\theoremstyle{definition}
\newtheorem{definition}[theorem]{Definition}                                       
\theoremstyle{remark}
\newtheorem{remark}[theorem]{Remark}

\DeclareMathAlphabet{\mathpzc}{OT1}{pzc}{m}{it}
\def\cp{[AS18]}

\def\co{\mathbbm{L}}

\newcommand{\lebesgue}{\boldsymbol{\lambda}}
\def\C{\mathcal{C}}

\def\CR{\mathpzc{CR}}
\def\l{L}
\def\E{\mathbb{E}}

\def\H{\mathbbm{H}}

\def\bdg{\operatorname{c}}
\def\K{\mathcal{K}}

\def\N{\mathbb{N}}
\def\p{\mathscr{P}}
\def\R{\mathbb{R}}
\def\A{\mathcal{A}}

\def\HH{\mathcal{H}}
\def\rd{\mathbbm{R}}
\def\V{\mathbb{V}}

\def\Ma{\mathbbm{M}}
\def\M{\mathds{M}}

\def\NN{\mathcal{N}}
\def\CC{\mathcal{C}}

\def\I{\mathbf{I}}\def\II{\mathbf{II}}\def\KL{\operatorname{KL}}
\def\N{\mathbb{N}}
\def\P{\mathbb{P}}

\def\supp{\operatorname{supp}}

\def\X{\mathbf{X}}

\def\tri{\widetriangle}

\renewcommand{\d}{\mathrm{d}}
\def\L{{\textrm L}}
\def\CR{\operatorname{CR}}
\def\Lip{\operatorname{Lip}}
\newcommand{\q}{^{\frac{p}{2}}}
\renewcommand{\p}{^{\frac{1}{p}}}
\newcommand{\e}{\mathrm{e}}

\newcommand{\FF}{\mathcal{F}}

\newcommand{\LL}{\mathcal{L}}

\renewcommand{\c}{^{\operatorname{c}}}
\newcommand{\VV}{\mathcal{V}}
\renewcommand{\tt}{^{t}}

\renewcommand{\hat}{\widehat}
\renewcommand{\tilde}{\widetilde}%

\newcommand{\T}{\mathcal{T}}

\usepackage{geometry}
\begin{document}
\title{$\operatorname{Sup}$-norm adaptive simultaneous drift estimation for ergodic diffusions
}

\author{Cathrine Aeckerle-Willems and Claudia Strauch\thanks{Universit\"at Mannheim, Institut f\"ur Mathematik, 68131 Mannheim, Germany.\newline
E-mail: aeckerle@uni-mannheim.de/strauch@uni-mannheim.de}}
\date{\vspace*{-1.5em}}

\maketitle

\begin{abstract}
We consider the question of estimating the drift and the invariant density for a large class of scalar ergodic diffusion processes, based on continuous observations, in $\sup$-norm loss. 
The unknown drift $b$ is supposed to belong to a nonparametric class of smooth functions of unknown order. 
We suggest an adaptive approach which allows to construct drift estimators attaining minimax optimal $\sup$-norm rates of convergence.
In addition, we prove a Donsker theorem for the classical kernel estimator of the invariant density and establish its semiparametric efficiency.
Finally, we combine both results and propose a fully data-driven bandwidth selection procedure which simultaneously yields both a rate-optimal drift estimator and an asymptotically efficient estimator of the invariant density of the diffusion. 
Crucial tool for our investigation are uniform exponential inequalities for empirical processes of diffusions.
\end{abstract}

\section{Introduction}
The field of nonparametric statistics for stochastic processes has become an integral part of statistics.
Due to their practical relevance as standard models in many areas of applied science such as genetics, meteorology or financial mathematics to name very few, the statistical analysis of diffusion processes receives special attention.
The first contribution of the present paper is an investigation of adaptive $\sup$-norm convergence rates for a nonparametric Nadaraya--Watson-type drift estimator, based on a continuous record of observations $(X_s)_{0\leq s\leq t}$ of a diffusion process on the real line. 
The suggested data-driven bandwidth choice relies on Lepski's method for adaptive estimation. 
Characterising upper and lower bounds, we show that the proposed estimation procedure in the asymptotic regime $t\to \infty$ is minimax rate-optimal over nonparametric H\"older classes. 
Remarkably, we impose only very mild conditions on the drift coefficient, not going far beyond standard assumptions that ensure the existence of ergodic solutions of the underlying SDE over the real line. 
In particular, we allow for unbounded drift coefficients.
Secondly, we prove a Donsker-type theorem for the classical kernel estimator of the invariant density in $\ell^\infty(\R)$ and establish its semiparametric efficiency. 
With regard to the direct relation between drift coefficient and the invariant density, it is clear that the corresponding estimation problems are closely connected. 
In a last step, we combine both tasks and suggest an adaptive bandwidth choice that simultaneously yields both an asymptotically efficient, asymptotically normal (in $\ell^\infty(\R)$) estimator of the invariant density and, at the same time, the corresponding minimax rate-optimal drift estimator. 

So far, results analysing the $\sup$-norm risk in the context of diffusion processes are rather scarce, even though quantifying expected maximal errors is of immense relevance, in particular for practical applications. 
We therefore start in the basic set-up of continuous observations of a scalar ergodic diffusion process. 
While the idealised framework of continuous observations of the process may be considered as being far from the reality, 
it is indisputably of substantial theoretical interest because the statistical results incorporate the very nature of the diffusion process, not being influenced by any discretisation errors.
Consequently, they serve as relevant benchmarks for further investigations. 
Moreover, our approach is attractive in the sense that it provides a reasonable starting point for extending the statistical analysis to discrete observation schemes and even multivariate diffusion processes. 
A second, very concrete motivation for our framework is the idea of bringing together methods from stochastic control and nonparametric statistics. 
Diffusion processes serve as a prototype model in stochastic optimal control problems which are solved under the long-standing assumption of continuous observations of a process driven by known dynamics. 
Relaxing this assumption to the framework of continuous observations of a process driven by an \emph{unknown} drift coefficient, imposing merely mild regularity assumptions, raises interesting questions on how to learn the dynamics by means of nonparametric estimation procedures and to control in an optimal way at the same time. With respect to the statistical methods, these applications typically require optimal bounds on $\sup$-norm errors. 
The present paper provides these tools for a large class of scalar diffusion processes. 

Taking a look at the evolution of the area of statistical estimation for diffusions up to the mid 2000's, we refer to \cite{goetal04} for a very nice summary. 
The monograph \cite{kut04} provides a comprehensive overview on  inference for one-dimensional ergodic diffusion processes on the basis of continuous observations considering pointwise and $L^2$-risk measures. \cite{ban78} is commonly mentioned as the first article addressing the question of nonparametric identification of diffusion processes from continuous data. 
In nonparametric models, asymptotically efficient estimators typically involve the optimal choice of a tuning parameter that depends on the smoothness of the nonparametric class of targets. 
From a practical perspective, this is not satisfying at all because the smoothness is usually not known. 
One thus aims at \emph{adaptive} estimation procedures which are based on purely data-driven estimators adapting to the unknown smoothness.

\cite{spok00} and \cite{dal05} were the first to study adaptive drift estimation in the diffusion model based on continuous observations. 
\cite{spok00} considers pointwise estimation whereas \cite{dal05} investigates a weighted $L^2(\R)$-norm. 
\cite{hoff99} initiated adaptive estimation in a high-frequency setting, proposing a data driven estimator of the diffusion coefficient based on wavelet thresholding which is rate optimal wrt $L^\gamma(D)$-loss, for $\gamma\in [1,\infty)$ and a compact set $D$. 
With regard to low-frequency data, we refer to the seminal paper by \cite{goetal04}. 
Their objective is inference on the drift and diffusion coefficient of diffusion processes with boundary reflections. 
The quality of the proposed estimators is measured in the $L^2([a,b])$ distance for any $0<a<b<1$. 
Like restricting to estimation on arbitrary but fixed compact sets, looking at processes with boundary reflections constitutes a possibility to circumvent highly technical issues that we will face in our investigation of diffusions on the entire real line. 
\citeauthor{goetal04} postulate that allowing diffusions on the real line would require to introduce a weighting in the risk measure given by the invariant density. 
This phenomenon will become visible in our results, as well. 
The same weighting function can be found in \cite{dal05}. 
Intuitively, it seems natural that the estimation risk would explode without a weighting since the observations of the continuous process during a finite period of time do not contain information about the behaviour outside the compact set where the paths lives in. 
For a more detailed heuristic account on the choice of the weight function for $L^2(\R)$-risk, we refer to Remark 4.1 in \cite{dal05}.
Sharp adaptive estimation of the drift vector for multidimensional diffusion processes from continuous observations for the $L^2$- and the pointwise risk has been addressed in \cite{str15} and \cite{str16}, respectively.

As illustrated, the pointwise and $L^2$-risks are already well-understood in different frameworks.
The present paper complements these developments by an investigation of the $\sup$-norm risk in the continuous observation scheme. 
In the low-frequency framework, this strong norm was studied in \cite{sotra16} who construct both an adaptive estimator of the drift and adaptive confidence bands. 
They prove a functional central limit theorem for wavelet estimators in a multi-scale space, i.e., considering a weaker norm that still allows to construct adaptive confidence bands for the invariant density and the drift with optimal $\ell^\infty([a,b])$-diameter. Still, there exist a lot of challenging open questions, and in view of the growing field of applications, there is a clear need for developing and adding techniques and tools for the statistical analysis of stochastic processes under $\sup$-norm risk. 
Ideally, these tools should include the probabilistic features of the processes and, at the same time, allow for an in-depth analysis of issues such as adaptive estimation in a possibly broad class of models. 

A common device for the derivation of adaptive estimation procedures in $\sup$-norm loss are uniform Talagrand-type concentration inequalities and moment bounds for empirical processes based on chaining methods. 
These tools are made available for a broad class of scalar ergodic diffusion processes in our recent paper \cite{cacs18}, in the sequel abbreviated as \cp. 
The concentration inequalities derived therein will serve as the central vehicle for our analysis, and we conjecture that they allow for generalizations on discrete observation schemes, multivariate state variables and even more general Markov processes. 
Therefore, the approach presented in this paper provides guidance for further statistical investigations of stochastic processes in $\sup$-norm risk.

Besides the frequentist statistical research, the Bayesian approach found a lot of interest, more recently. 
In the framework of continuous observations, \cite{meuletal06} consider the asymptotic behaviour of posterior distributions in a general Brownian semimartingale model which, as a special case, includes the ergodic diffusion model.
\cite{poketal13} investigate a Bayesian approach to nonparametric estimation of the periodic drift of a scalar diffusion from continuous observations and derive bounds on the rate at which the posterior contracts around the true drift in $L^2([0,1))$-norm. 
Improvements in terms of these convergence rates results and adaptivity are given in \cite{vwvz16}.
Nonparametric Bayes procedures for estimating the drift of one-dimensional ergodic diffusion models from discrete-time low-frequency data are studied in \cite{vdmvz13}.
The authors give conditions for posterior consistency and verify these conditions for concrete priors.
Given discrete observations of a scalar reflected diffusion, \cite{niso17} derive (and verify) conditions in the low-frequency sampling regime for prior distributions on the diffusion coefficient and the drift function that ensure minimax optimal contraction rates of the posterior distribution over H\"older--Sobolev smoothness classes in $L^2([a,b])$-distance, for any $0<a<b<1.$


%
%
%
%
%
%
%


\subsection*{Basic framework and outline of the paper}
Taking into view the $\sup$-norm risk, the aim of this paper is to suggest a rate-optimal nonparametric drift estimator, based on continuous observations of an ergodic diffusion process on the real line which is given as the solution of the SDE
\begin{equation}\label{SDE}
\d X_t\ =\ b(X_t)\d t+ \sigma(X_t)\d W_t, \quad X_0=\xi, \ t>0,
\end{equation}
with unknown drift function $b\colon\R\to\R$, dispersion $\sigma\colon\R\to (0,\infty)$ and some standard Brownian motion $W=(W_t)_{t\geq0}$.
The initial value $\xi$ is a random variable independent of $W$. 
We restrict to the ergodic case where the Markov process $(X_t)_{t\geq 0}$ admits an invariant measure, and we denote by $\rho_b$ and $\mu_b$ the invariant density and the associated invariant measure, respectively. 
Furthermore, we will always consider stationary solutions of \eqref{SDE}, i.e., we assume $\xi\sim\mu_b$.

In the set-up of continuous observations, there is no interest in estimating the volatility $\sigma^2$ since this quantity is identifiable using the quadratic variation of $X$. We thus focus on recovering the unknown drift.
We develop our results in the following classical scalar diffusion model.

\begin{definition}\label{def:B}
Let $\sigma\in \operatorname{Lip}_{\operatorname{loc}}(\R)$ and assume that, for some constants $\nu\geq 0$, $\C\geq 1$, $\sigma$ satisfies $ |\sigma(x)| \leq\C(1+|x|) $ and $ |\sigma^{-1}(x)| \leq\C(1+|x|^\nu) $ for all $x\in\R$ . For fixed constants $A,\gamma>0$ and $\C\geq 1$, define the set $\Sigma=\Sigma(\C,A,\gamma,\sigma)$ as
\begin{equation}\label{eq:sigma}
\Sigma\ :=\ \Big\{b \in \operatorname{Lip}_{\operatorname{loc}}(\R)\colon|b(x)| \leq\C(1+|x|),\ \forall|x|>A\colon \frac{b(x)}{\sigma^2(x)}\operatorname{sgn}(x)\leq -\gamma\Big\}.
\end{equation}
\end{definition}

Given any $b\in\Sigma$, there exists a unique strong solution of the SDE \eqref{SDE} with ergodic properties and invariant density
\begin{equation}\label{eq:invdens}
\rho(x)=\rho_b(x)\ :=\ \frac{1}{C_{b,\sigma}\sigma^2(x)}\ \exp\left(\int_0^x\frac{2b(y)}{\sigma^2(y)}\d y\right),\quad x\in\R,
\end{equation}
with $C_{b,\sigma}:=\int_\R\frac{1}{\sigma^2(u)}\ \exp\left(\int_0^u\frac{2b(y)}{\sigma^2(y)}\d y\right)\d u$ denoting the normalising constant. 
Throughout the sequel and for any $b\in\Sigma$, we will denote by $\E_b$ the expected value with respect to the law of $\X$ associated with the drift coefficient $b$.
The distribution function corresponding to $\rho=\rho_b$ and the invariant measure of the distribution will be denoted by $F=F_b$ and $\mu=\mu_b$, respectively. 

Our statistical analysis relies heavily on uniform concentration inequalities for continuous-time analogues of empirical processes of the form $t^{-1}\int_0^t f(X_s)\d s - \E_b[f(X_0)]$, $f\in\mathcal F$, as well as stochastic integrals $t^{-1}\int_0^t f(X_s)\d X_s - \E_b[b(X_0)f(X_0)]$, $f\in\mathcal F$, indexed by some infinite-dimensional function class $\mathcal F$. 
These key devices are provided in our work on concentration inequalities for scalar ergodic diffusions. 
They are tailor-made for the investigation of $\sup$-norm risk criteria and can be considered as continuous-time substitutes for Talagrand-type concentration inequalities and moment bounds for empirical processes in the classical i.i.d.~framework. 
In \cp, upper bounds on the expected $\sup$-norm error for a kernel density estimator of the invariant density (that we will use in the present work) are derived as a first statistical application of the developed concentration inequalities. 
In Section \ref{sec:pre}, we will present the announced probabilistic tools and statistical results from \cp~that will be of crucial importance in our subsequent developments.
The advantage of the methods proposed in \cp~is that the martingale approximation approach - which is at the heart of the derivations - yields very elementary simple proofs, working under minimal assumptions on the diffusion process.

\paragraph{The estimators}
Given continuous observations $X^{t}=(X_s)_{0\leq s\leq t}$ of a diffusion process as described in Definition \ref{def:B}, first basic statistical questions concern the estimation of the invariant density $\rho_b$ and the drift coefficient $b$ and the investigation of the respective convergence properties. 
Since $b=(\rho_b\sigma^2)'/(2\rho_b)$ (for differentiable $\sigma$), the question of drift estimation is obviously closely related to estimation of the invariant density $\rho_b$ and its derivative $\rho_b'$. 
For some smooth kernel function $K\colon \R\to\R^+$ with compact support, introduce the standard kernel invariant density estimator
\begin{equation}\label{est:dens}
\rho_{t,K}(h)(x)\ :=\ \frac{1}{th}\int_0^tK\left(\frac{x-X_u}{h}\right)\d u,\quad x\in\R.
\end{equation}
A natural estimator of the drift coefficient $b\in\Sigma(\C,A,\gamma,1)$, which relies on the analogy between the drift estimation problem and the model of regression with random design, is given by a Nadaraya--Watson-type estimator of the form  
\begin{align}\label{est:drift}
b_{t,K}(h)(x)&:=\ \frac{\overline\rho_{t,K}(h)(x)}{\rho_{t,K}(t^{-1/2})(x) + \sqrt{\frac{\log t}{t}}\exp\left(\sqrt{\log t}\right)},\\\label{est:derivative}
\text{where } 
\overline\rho_{t,K}(h)(x)&:=\ \frac{1}{th}\int_0^tK\left(\frac{x-X_s}{h}\right)\d X_s.
\end{align}
We recognize the kernel density estimator in the denominator, and we will see that $\overline\rho_{t,K}$ with the proposed (adaptive) bandwidth choice serves as a rate-optimal estimator of $b\rho_b$. The additive term in the denominator prevents it from becoming small too fast in the tails. 

Given a record of continuous observations of a scalar diffusion process $\X$ with coefficients as described in Definition \ref{def:B}, the local time estimator $\rho_0(\bullet):=t^{-1}L_t^\bullet (\X)$, for $(L_t^a(\X),t\geq 0, a\in\R)$ denoting the local time process of $\X$, is available. 
This is a natural density estimator since diffusion local time can be interpreted as the derivative of the empirical measure. 
In the past, the latter was exhaustively studied for pointwise estimation and in $L^2$-risk unlike the $\sup$-norm case. 
In \cite[Sec.~7]{kutoyants98}, weak convergence of the local time estimator to a Gaussian process in $\ell^\infty(\R)$ is shown.
The same is done for more general diffusion processes in \cite{vdvvz05}. 
Having provided the required tools from empirical process theory, upper bounds on all moments of the $\sup$-norm error of $\rho_0$ are proven in \cp.
Unfortunately, the local time estimator is viewed as not being very feasible in practical applications.
In addition, it does not offer straightforward extensions to the case of discretely observed or multivariate diffusions, in sharp contrast to the classical kernel-based density estimator. 
We therefore advocate the usage of the kernel density estimator introduced in \eqref{est:dens} which can be viewed as a universal approach in nonparametric statistics, performing an optimal behaviour over a wide range of models. 
Furthermore, the kernel density estimator naturally appears in the denominator of our Nadaraya--Watson-type drift estimator defined according to \eqref{est:drift}.

\paragraph{Asymptotically efficient density estimation}
In the present work, we will complement the sup-norm analysis started in \cp~with an investigation of the asymptotic distribution of the kernel density estimator in a functional sense. 
We will prove a Donsker-type theorem for the kernel density estimator, thereby demonstrating that this estimator for an appropriate choice of bandwidth behaves asymptotically like the local time estimator. 
We then go one step further and establish optimality of the limiting distribution, optimality seen in the sense of the general convolution theorem 3.11.2 for the estimation of Banach space valued parameters presented in \cite{vdvw96}. 
Their theorem states that, for an asymptotically normal sequence of experiments and any regular estimator, the limiting distribution is the convolution of a specific Gaussian process and a noise factor. 
This Gaussian process is viewed as the optimal limit law, and we refer to it as the \emph{semiparametric lower bound}. 
We establish this lower bound and verify that it is achieved by the kernel density estimator. 
The Donsker-type theorem and the verification of semiparametric efficiency of the kernel-based estimator are the main results on density estimation in the present paper. 
They are presented in Section \ref{sec:3}. 
Donsker-type theorems can be regarded as frequentist versions of functional Bernstein--von Mises theorems to some extent. 
In particular, our methods and techniques are interesting for both the frequentist and Bayesian community. 
The optimal limiting distribution in the sense of the convolution theorem is relevant in the context of Bayesian Bernstein--von Mises theorems in the following sense: 
If this lower bound is attained, Bayesian credible sets are optimal asymptotic frequentist confidence sets as argued in \cite{castillo14}; see also \cite[p.~12]{nicklsoehl17} who address Bernstein--von Mises theorems in the context of compound Poisson processes. 
Our approach concerning the question of efficiency is based on some recent work by Nickl and Ray on a Bernstein--von Mises theorem for multidimensional diffusions. 
We thank Richard Nickl for the private communication that motivated the derivation of the semiparametric lower bound in this work. 

\paragraph{Minimax optimal adaptive drift estimation in $\sup$-norm}
Subject of Section \ref{sec:adapt_einfach} is an adaptive scheme for the $\sup$-norm rate-optimal estimation of the drift coefficient.
This is the main contribution and initial motivation of the present paper.
Our approach for estimating the drift coefficient is based on Lepski's method for adaptive estimation and the exponential inequalities presented in Section \ref{sec:pre}. 
For proving upper bounds on the expected $\sup$-norm loss, we follow closely the ideas developed in \cite{gini09} for the estimation of the density and the distribution function in the classical i.i.d.~setting. 
We suggest a purely data-driven bandwidth choice $\tri h_T$ for the estimator $b_{T,K}$ defined in \eqref{est:drift} and derive upper bounds on the convergence rate of the expected $\sup$-norm risk uniformly over H\"older balls in Theorem \ref{theo:est}, imposing very mild conditions on the drift coefficient.
To establish minimax optimality of the rate, we prove lower bounds presented in Theorem \ref{thm:lowerboundableitung}. 

\paragraph{Simultaneous adaptive density and drift estimation}
Observing from \eqref{eq:invdens} that the invariant density is a transformation of the integrated drift coefficient, it is not surprising that we can carry over the aforementioned approach in \cite{gini09} (which aims at simultaneous estimation of the distribution function and density in the i.i.d.~framework) to the problems of invariant density and drift estimation. 
We suggest a simultaneous bandwidth selection procedure that allows to derive a result in the spirit of their Theorem 2. Adjusting the procedure from Section \ref{sec:adapt_einfach} for choosing the bandwidth $\tri h_t$ in a data-driven way, we can find a bandwidth $\hat h_t$ such that $\rho_{t,K}(\hat h_t)$ is an asymptotically efficient estimator in $\ell^\infty(\R)$ for the invariant density and, at the same time, $b_{t,K}(\hat h_t)$ estimates the drift coefficient with minimax optimal rate of convergence wrt $\sup$-norm risk. 
We formulate this result in Theorem \ref{theo:sim}.

\section{Preliminaries}\label{sec:pre}
We will investigate the question of adapting to unknown H\"older smoothness. 
For ease of presentation, we will suppose in the sequel that $\sigma\equiv1$. 
The subsequent results however can be extended to the case of a general diffusion coefficient fulfilling standard regularity and boundedness assumptions.
Recall the definition of the class $\Sigma=\Sigma(\CC,A,\gamma,\sigma)$ of drift functions in \eqref{eq:sigma}.


\begin{definition}\label{def:Sigma}
Given $\beta,\mathcal{L}>0$, denote by $\mathcal{H}_{\R}(\beta,\mathcal{L})$ the \emph{H\"older class} (on $\R$) as the set of all functions $f\colon\R\to\R$ which are $l:=\lfloor \beta \rfloor$-times differentiable and for which
\begin{align*}
\|f^{(k)}\|_{\infty}&\leq\ \mathcal{L}  \qquad\qquad\forall\, k=0,1,...,l,\\
\|f^{(l)}(\cdot+s)-f^{(l)}(\cdot)\|_{\infty}&\leq\ \mathcal{L}|s|^{\beta- l}  \qquad \forall\, s\in\R.
\end{align*}
Set 
\begin{equation}\label{def:Sigmabeta}
\Sigma(\beta,\mathcal{L})=\Sigma(\beta,\LL,\mathcal C,A,\gamma)\ :=\ \left\{b\in\Sigma(\CC,A,\gamma,1)\colon\ \rho_b\in\mathcal{H}_\R(\beta+1,\mathcal{L})\right\}.
\end{equation}
Here, $\lfloor \beta \rfloor$ denotes the greatest integer strictly smaller than $\beta$.
\end{definition}

Considering the class of drift coefficients $\Sigma(\beta,\mathcal{L})$, we use kernel functions satisfying the following assumptions,
\begin{equation}\label{kernel}
\begin{array}{r@{}l}
&{}\bullet\quad K:\R\rightarrow \R^+ \text{ is Lipschitz continuous and symmetric};\\ [3pt]
&{}\bullet\quad\supp(K)\subseteq [-1/2,1/2];\\[3pt]
&{}\bullet\quad \text{for some } \alpha\geq \beta +1, K\text{ is of order }\lfloor\alpha \rfloor.
\end{array}
\end{equation}
The subsequent deep results from \cp~are fundamental for the investigation of the $\sup$-norm risk. 
They rely on diffusion specific properties, in particular the existence of local time, on the one hand, and classical empirical process methods like the generic chaining device on the other hand. 
In the classical setting of statistical inference based on i.i.d.~observations $X_1,...,X_n$, the analysis of $\sup$-norm risks typically requires investigating empirical processes of the form 
$\left(n^{-1}\sum_{i=1}^n f(X_i)\right)_{f\in\mathcal F}$, indexed by a possibly infinite-dimensional class $\mathcal F$ of functions which, in many cases, are assumed to be uniformly bounded.
Analogously, in the current continuous, non-i.i.d.~setting, our analysis raises questions about empirical processes of the form 
\[
	\left(\frac{1}{t}\int_0^t f(X_s) \d s\right)_{f\in\mathcal F}\quad\text{ and, more generally, }\quad\left(\frac{1}{t}\int_0^t f(X_s) \d X_s\right)_{f\in\mathcal F}.
\]
Clearly, the finite variation part of the stochastic integral entails the need to look at unbounded function classes since we do not want to restrict to bounded drift coefficients.
Answers are given in \cp~where we provide exponential tail inequalities both for 
\[
\sup_{f\in \mathcal F}\ \left[\frac{1}{t}\int_0^t f(X_s) \d s\ -\ \E_b\left[f(X_0)\right]\right]\quad\text{and}\quad\sup_{f\in \mathcal F}\ \left[\frac{1}{t}\int_0^t f(X_s) \d X_s\ -\ \E_b\left[b(X_0)f(X_0)\right]\right] ,
\]
imposing merely standard entropy conditions on $\mathcal F$.
As can be seen from the construction of the estimators, we have to exploit these results in order to deal with both empirical diffusion processes induced by the kernel density estimator $\rho_{t,K}(h)$ (see \eqref{est:dens}) and with stochastic integrals like the estimator $\overline\rho_{t,K}(h)$ of the derivative of the invariant density (see \eqref{est:derivative}).
One first crucial auxiliary result for proving the convergence properties of the estimation schemes proposed in Sections \ref{sec:adapt_einfach} and \ref{sec:sim} is stated in the following

\begin{proposition}[Concentration of the estimator $\overline{\rho}_{t,K}(h)$ of $\rho_b'/2$]\label{prop:csi}
Given a continuous record of observations $X^t=(X_s)_{0\leq s\leq t}$ of a diffusion $\X$ with $b\in\Sigma=\Sigma(\C,\A,\gamma,1)$ as introduced in Definition \ref{def:B} and a kernel $K$ satisfying \eqref{kernel}, define the estimator $\overline\rho_{t,K}$ according to \eqref{est:derivative}.
Then, there exist constants $\co, \tilde\co_0$ such that, for any $u,p\geq 1$, $h\in (0,1)$, $t\geq 1$,
\begin{align}
\begin{split}\label{con_stoch_int}
\sup_{b\in \Sigma} \left(\E_b\left[\|\overline{\rho}_{t,K}(h)-\E_b\left[\overline{\rho}_{t,K}(h)\right]\|_\infty^p\right]\right)\p&\leq\ \phi_{t,h}(p),\\
\sup_{b\in \Sigma}\P_b\left(\sup_{x\in\R}\left|\overline\rho_{t,K}(h)(x)-\E_b\left[\overline \rho_{t,K}(h)(x)\right]\right|>\e\phi_{t,h}(u)\right)&\leq\ \e^{-u},
\end{split}
\end{align}
for 
\begin{align}
\begin{split}\label{phith}
\phi_{t,h}(u)&:=\ \co \Bigg\{ 
\frac{1}{\sqrt t}\Big\{\left(\log \left(\frac{ut}{h}\right)\right)^{3/2}+\left(\log\left(\frac{ut}{h}\right)\right)^{1/2}+u^{3/2}\Big\}+\frac{u}{th}+  \frac{1}{h} \exp\left(-\tilde\co_0 t \right)\\
&\qquad +\frac{1}{\sqrt{th}}\left(\log\left(\frac{ut}{h}\right)\right)^{1/2} + \frac{1}{t^{3/4}\sqrt h}\log\left(\frac{ut}{h}\right) + \frac{1}{\sqrt{th}} \left\{\sqrt u+ \frac{u}{t^{1/4}}\right\}\Bigg\}.
\end{split}
\end{align}
\end{proposition}
\begin{proof}
We apply Theorem 18 in \cp~to the class
\begin{equation}\label{functionclass}
\mathcal F:= \left\{K\left(\frac{x-\cdot}{h}\right):\,x\in\mathbb Q\right\}.
\end{equation}
For doing so, note that $\sup_{f\in\mathcal F}\|f\|_\infty\leq \|K\|_\infty$, and, for $\lebesgue$ denoting the Lebesgue measure,  
\begin{eqnarray*}
\left\|K\left(\frac{x-\cdot}{h}\right)\right\|^2_{L^2(\lebesgue)}&=& \int K^2\left(\frac{x-y}{h}\right)\d y\ =\ h \int K^2(z)\d z\ \leq\ h\|K \|^2_{L^2(\lebesgue)}
\end{eqnarray*} 
and $\sup_{f\in\mathcal F}\lebesgue(\supp(f))\leq h$. 
Due to the Lipschitz continuity of $K$, Lemma 23 in \cp~yields constants $\mathbbm A>0$, $v\geq 2$ (only depending on $K$) such that, for any probability measure $\mathbb Q$ on $\R$ and any $0<\epsilon<1$, $N(\epsilon,\mathcal F,\|\cdot\|_{L^2(\mathbb Q)})\leq (\mathbbm A/\epsilon)^v$.
Here and throughout the sequel, given some semi-metric $d$, $N(u,\FF,d)$, $u>0$, denotes the covering number of $\FF$ wrt $d$, i.e., the smallest number of balls of radius $u$ in $(\FF,d)$ needed to cover $\FF$. 
Since the assumption on the covering numbers of $\mathcal F$ in Theorem 18 in \cp~is fulfilled, Theorem 18 can be applied to $\FF$ with $\mathcal S := h\ \max\{\|K \|^2_{L^2(\lebesgue)},1\}$ and $\V:= \sqrt h\|K \|_{L^2(\lebesgue)}$.
In particular, there exist positive constants $\tilde\co$, $\Lambda$, $\tilde\co_0$ and $\co$ such that 
\begin{align*}
&\sup_{b\in \Sigma} \left(\E_b\left[\|\overline{\rho}_{t,K}(h)-\E_b\left[\overline{\rho}_{t,K}(h)\right]\|^p_\infty\right]\right)\p\\
&\qquad \leq\ \tilde\co\Bigg\{ 
\frac{1}{\sqrt t}\Big\{\left(\log\left(\sqrt{\frac{h + p\Lambda t}{h}}\right)\right)^{3/2} + \left(\log\left(\sqrt{\frac{h + p\Lambda t}{h}}\right)\right)^{1/2} + p^{3/2}\Big\}\\
&\hspace*{8em} +\frac{p}{th} + \frac{1}{h} \exp\left(-\tilde\co_0 t \right)
+ \frac{1}{\sqrt{th}} \left(\log\left(\sqrt{\frac{h + p\Lambda t}{h}}\right)\right)^{1/2}\\
&\hspace{8em} \, + \, \frac{1}{t^{3/4}\sqrt h}\left(1+\log\left(\sqrt{\frac{h + p\Lambda t}{h}}\right)\right) + \frac{1}{\sqrt{th}} \left\{\sqrt p+ \frac{p}{t^{1/4}}\right\}\Bigg\}\\
&\qquad\leq \phi_{t,h}(p),
\end{align*}
and \eqref{con_stoch_int} immediately follows.
\end{proof}


The uniform concentration results for stochastic integrals from \cp~further allow to prove the following result on the $\sup$-norm distance $\|t^{-1}L_t^\bullet(X)-\rho_{t,K}(h)\|_\infty$ between the local time and the kernel density estimator. 
The exponential inequality for this distance will be the key to transferring the Donsker theorem for the local time to the kernel density estimator. 
It can also be interpreted as a result on the uniform approximation error of the scaled local time by its smoothed version, noting that $\rho_{t,K}(h)$ can be seen as a convolution of a mollifier and a scaled version of diffusion local time.
The next result actually parallels Theorem 1 in \cite{gini09} which states a subgaussian inequality for the distribution function in the classical i.i.d.~set-up. 
It serves as an important tool for the analysis of the proposed adaptive scheme for simultaneous estimation of the distribution function and the associated density in \cite{gini09}.
The subsequent proposition plays an analogue role for the adaptive scheme for simultaneous estimation of the invariant density and the drift coefficient presented in Section \ref{sec:sim}.

\begin{proposition}[Theorem 15 in \cp]\label{theo:cath1}
Given a diffusion $\X$ with $b\in\Sigma(\beta,\mathcal L)$, for some $\beta,\mathcal L >0$,
consider some kernel function $K$ fulfilling \eqref{kernel} and $h=h_t\in(0,1)$ such that $h_t\geq t^{-1}$. 
Then, there exist positive constants $\VV,\xi_1,\Lambda_0,\Lambda_1$ and $\co$ such that, for all $\lambda \geq \lambda_0(h)$, where
\begin{align*}
\lambda_0(h)& := 8\Lambda_0\bigg[ \sqrt h \VV\e\mathbb L\left\{1+ \log\left(\frac{1}{\sqrt h \VV}\right)+\log t\right\}+\e\co\sqrt t \exp(-\xi_1 t)\\
&\hspace*{14em}\ + \ \sqrt t h^{\beta +1}\frac{\mathcal L}{2\lfloor \beta + 1\rfloor!}\int |K(v)v^{\beta +1 }|\d v\bigg],
\end{align*}
and any $t>1$,
\[
\sup_{b\in\Sigma(\beta,\mathcal L)}\P_b\left(\sqrt t\Big\|\rho_{t,K}(h)-\frac{L_t^\bullet(X)}{t}\Big\|_\infty>\lambda\right)\ \leq \ \exp\left(-\frac{\Lambda_1\lambda}{\sqrt h}\right).
\]
\end{proposition}

The very first step of our approach to $\sup$-norm adaptive drift estimation consists in estimating the invariant density in $\sup$-norm loss.
Corresponding upper bounds on the $\sup$-norm risk have been investigated in \cp.
We next cite these bounds for the local time estimator and the kernel density estimator. 
Our estimation procedure does \emph{not} involve the local time density estimator.
For the sake of presenting a complete statistical $\sup$-norm analysis of ergodic scalar diffusions based on continuous observations, we still include it here.

\begin{lemma}[Moment bound on the supremum of centred diffusion local time, Corollary 16 of \cp]\label{lem:difflt}
Let $\X$ be as in Definition \ref{def:B}.
Then, there are positive constants $\zeta,\zeta_1$ such that, for any $p,t\geq 1$, 
\begin{align*}
\sup_{b\in\Sigma(\C,A,\gamma,1)}\left(\E_b\left[\Big\|\frac{L_t^\bullet(\X)}{t}-\rho_b\Big\|_\infty^p\right]\right)\p\ \leq\ \zeta\left(\frac{p}{t}+\frac{1}{\sqrt t}\left(1+\sqrt p+\sqrt{\log t}\right)+t\e^{-\zeta_1 t}\right).
\end{align*}
\end{lemma}

In \cp, we have also shown the analogue fundamental result for the $\sup$-norm risk of the kernel density estimator.
The following upper bounds will be essential for deriving convergence rates of the Nadaraya--Watson-type drift estimator (see \eqref{est:drift}).

\begin{proposition}[Concentration of the kernel invariant density estimator, Corollary 14 of \cp]\label{prop:con_kernel_density_estimator}
Let $\X$ be a diffusion with $b\in\Sigma(\beta,\mathcal L)$, for some $\beta,\mathcal L>0$, and let $K$ be a kernel function fulfilling \eqref{kernel}.
Given some positive bandwidth $h$, define the estimator $\rho_{t,K}(h)$ according to \eqref{est:dens}. 
Then, there exist positive constants $\nu_1,\nu_2,\nu_3$ such that, for any $p,u\geq 1$, $t>0$,
\begin{align}\label{eq:prop6} 
\sup_{b\in\Sigma(\beta,\mathcal L)}\left(\E_b\left[\left\|\rho_{t,K}(h) - \rho_b\right\|^p_\infty\right]\right)\p
&\leq\ \psi_{t,h}(p),\\\nonumber
\sup_{b\in\Sigma(\beta,\mathcal L)}\P_b\left(\left\|\rho_{t,K}(h)-\rho_b\right\|_\infty\geq\e\psi_{t,h}(u)\right)&\leq\ \e^{-u},
\end{align}
for 
\begin{align}
\begin{split}\label{def:psi}
\psi_{t,h}(u)&:=\ \frac{\nu_1}{\sqrt t}\left\{1+\sqrt{\log\left(\frac{1}{\sqrt h}\right)} + \sqrt{\log(ut)}+\sqrt u\right\}+\frac{\nu_2u}{t}+\frac{1}{h}\e^{-\nu_3 t} \\
&\hspace*{14em}+\frac{\mathcal Lh^{\beta +1}}{\lfloor \beta+1\rfloor!}\int |v^{\beta +1}K(v)|\d v.
\end{split}
\end{align}
\end{proposition}
Specifying to $h=h_t\sim t^{-1/2}$, an immediate consequence of \eqref{eq:prop6} is the convergence rate $\sqrt{\log t/t}$ for the risk $\sup_{b\in\Sigma(\beta,\mathcal L)}\E_b\left[\left\|\rho_{t,K}(h) - \rho_b\right\|_\infty\right]$ of the kernel density estimator $\rho_{t,K}(t^{-1/2})$.
Note that we obtain the parametric convergence rate for the bandwidth choice $t^{-1/2}$ which in particular does not depend on the (typically unknown) order of smoothness of the drift coefficient. 
Thus, there is no extra effort needed for adaptive estimation of the invariant density. This phenomenon appears only in the scalar setting.

\bigskip

\section{Donsker-type theorems and asymptotic efficiency of kernel invariant density estimators}\label{sec:3}
This section is devoted to the study of weak convergence properties of the kernel density estimator $\rho_{t,K}$.
Using the exponential inequality for $\|t^{-1}L^\bullet_t(X)-\rho_{t,K}(h)\|_\infty$ (Proposition \ref{theo:cath1} from Section \ref{sec:pre}), we derive a uniform CLT for the kernel invariant density estimator.
In particular, the result holds for the `universal' bandwidth choice $h\sim t^{-1/2}$.
Furthermore, we use the general theory developed in \cite{vdvw96} for establishing asymptotic semiparametric efficiency of $\rho_{t,K}(t^{-1/2})$ in $\ell^\infty(\R)$.

\subsection{Donsker-type theorems}
The exponential inequality for the $\sup$-norm difference of the kernel and the local time density estimator stated in Proposition \ref{theo:cath1} allows to transfer an existing Donsker theorem for the local time density estimator presented in \cite{vdvvz05}.

\begin{proposition}\label{donskerdensity}
Given a diffusion $\X$ with $b\in\Sigma(\beta,\LL)$, consider some kernel function $K$ fulfilling \eqref{kernel}. 
Define the estimator $\rho_{t,K}(h)$ according to \eqref{est:dens} with bandwidth $h=h_t\in [t^{-1},1)$ satisfying $\sqrt t h^{\beta+1}\rightarrow 0$, as $t\to \infty.$
Then,
\[
\sqrt t\left(\rho_{t,K}(h)-\rho_b\right)\ \stackrel{\P_b}{\Longrightarrow}\ \H,\qquad \text{ as } t\to\infty,
\]
in $\ell^\infty(\R)$, where $\H$ is a centered, Gaussian random map with covariance structure
\begin{equation}\label{cov:H}
\E[\H(x)\H(y)]=4m(\R)\rho_b(x)\rho_b(y)\int_\R (\mathds{1}\{[x,\infty)\}-F_b)(\mathds{1}\{[y,\infty)\}-F_b)\d s,
\end{equation}
$m$ and $s$ denoting the speed measure and the scale function of $\X$, respectively.
\end{proposition}

\begin{proof}
We apply Proposition \ref{theo:cath1} to show that 
\begin{equation}\label{convinprop}
\sqrt t\Big\|\rho_{t,K}(h)-\frac{L_t^\bullet(X)}{t}\Big\|_\infty\ =\ o_{\P_b}(1).
\end{equation}
There exists a constant $C>0$ such that 
$\lambda_t:=C\left(\sqrt h \left( 1+ \log t\right) + \sqrt t h^{\beta + 1}\right)$ 
fulfills the assumption $\lambda_t\geq \lambda_0(h)$, for any bandwidth $h=h_t$ satisfying the above conditions.
Since $\lambda_t=o(1)$, for any $\epsilon>0$ and $t$ sufficiently large,
\begin{align*}
\P_b\left(\sqrt t\|\rho_{t,K}(h)-t^{-1}L_t^\bullet(X)\|_\infty>\epsilon\right)
&\leq \ \P_b\left(\sqrt t\|\rho_{t,K}(h)-t^{-1}L_t^\bullet(X)\|_\infty>\lambda_t\right)\\
&\leq \ \exp\left(-\Lambda_1\frac{\lambda_t }{\sqrt h}\right)\\
&=\ \exp\left(-\Lambda_1C\left(\left( 1+ \log t\right) + \sqrt t h^{\beta + \frac{1}{2}}\right)\right)\\
&\longrightarrow\ 0, \quad \text{as }t\to\infty.
\end{align*}
Consequently, \eqref{convinprop} holds, and Lemma \ref{lem:gaussexists} from Section \ref{appeff} gives the assertion.
\end{proof}

\begin{remark}
Donsker-type results turn out to be useful far beyond the question of the behaviour of the density estimator wrt the $\sup$-norm as a specific loss function. 
In particular, they provide immediate access to solutions of statistical problems concerned with functionals of the invariant density $\rho_b$.
Clearly, this includes the estimation of bounded, linear functionals of $\rho_b$ such as integral functionals, to name just one common class. 
As an instance, \cite{kutoyants07} study the estimation of moments $\mu_b(G)$ for known functions $G$. 
The target $\mu_b(G)$ is estimated by the empirical moment estimator $t^{-1}\int_0^t G(X_s)\d s$, and it is shown that this estimator is asymptotically efficient in the sense of local asymptotic minimaxity (LAM) for polynomial loss functions.
Parallel results can directly be deduced from the Donsker theorem. 
Defining the linear functional $\Phi_G\colon\ell^\infty(\R)\to \R$, $h\mapsto \int h(x)G(x)\d x$, the target can be written as $\Phi_G(\rho_b)$, and the empirical moment estimator equals the linear functional applied to the local time estimator, that is, 
\[\frac{1}{t}\int_0^t G(X_s)\d s\ =\ \Phi_G(t^{-1}L^\bullet_t(X)).\] 
Thus, if $\Phi_G$ is bounded, it follows from the results of \cite{vdvvz05} (see Lemma \ref{lem:gaussexists} in Section \ref{appeff}) and from Proposition \ref{donskerdensity}, respectively, that 
\[\sqrt t\left(\Phi_G(t^{-1}L_t^\bullet(X))-\Phi_G(\rho_b) \right)\text{ as well as }\sqrt t\left(\Phi_G(\rho_{t,K}(t^{-1/2}))-\Phi_G(\rho_b) \right)\] 
are asymptotically normal with the limiting distribution $\Phi_G(\H)$.
Optimality of $\Phi_G(\H)$ in the sense of the convolution theorem 3.11.2 in \cite{vdvw96} will be shown in the next section. 
\end{remark}

Not only linear, but also nonlinear functions that allow for suitable linearisations can be analysed, once the required CLTs and optimal rates of convergence are given.
This is related to the so-called \emph{plug-in property} introduced in \cite{biri03}. 
The suggested connection is explained a bit more detailed in \cite{gini09}.

\subsection{Semiparametric lower bounds for estimation of the invariant density}\label{subsec:crinv}
We now want to analyse semiparametric optimality aspects of the limiting distribution in Proposition \ref{donskerdensity} as treated in Chapter 3.11 in \cite{vdvw96} or Chapter 25 of \cite{vdv00}. 
To this end, we first look at lower bounds.

Denote by $\P_{t,h}$ the law of a diffusion process $Y^t:=(Y_s)_{0\leq s\leq t}$ with perturbed drift coefficient $b+t^{-1/2}h$, given as a solution of the SDE
\[\d Y_s\ =\ \left(b(Y_s)+\frac{h(Y_s)}{\sqrt t}\right)\d s + \d W_s,\quad Y_0=X_0,\]
and denote by $\rho_{b+t^{-1/2}h}$ the associated invariant density. 
Set $\P_b:=\P_{t,0}$, and define the set of experiments 
\begin{equation}\label{def:set}
\left\{C(0,t),\mathcal B\left(C(0,t)\right),\P_{t,h}\colon h\in G\right\},\quad t>0,
\end{equation}
with $G=\ell^\infty(\R)\cap \text{Lip}_{\operatorname{loc}}(\R)$ viewed as a linear subspace of $\l^2(\mu_b)$.
By construction and Girsanov's Theorem (cf.~\cite[Theorem 7.18]{lipshi01}), the log-likelihood is given as
\begin{align*}
\log\left(\frac{\d\P_{t,h}}{\d\P_b}\right)(X^t)&=\ \frac{1}{\sqrt t}\int_0^t h(X_s)\d W_s-\frac{1}{2t}\int_0^th^2(X_s)\d s\\
&=\ \Delta_{t,h}-\frac{1}{2}\|h\|_{L^2(\mu_b)}^2+o_{\P_b}(1),
\end{align*}
where $\Delta_{t,h}:=t^{-1/2}\int_0^Th(X_s)\d W_s$.
Here, the last line follows from the law of large numbers for ergodic diffusions, and the CLT immediately gives 
$\Delta_{t,h}\stackrel{\P_b}{\Longrightarrow}\NN(0,\|h\|_{L^2(\mu_b)}^2)$.
Thus, \eqref{def:set} is an asymptotically normal model.
Lemma \ref{lemmadifferentiability} from Section \ref{appeff} now implies that the sequence $\Psi(\P_{t,h}):=\rho_{b+t^{-1/2}h}$, $t>0,$ is \emph{regular} (or \emph{differentiable}).
In fact, it holds
\begin{equation}\label{differentiability}
\sqrt T(\Psi(\P_{t,h})-\Psi(\P_b))\ \longrightarrow_{t\to\infty}\ A'h\quad\text{ in }\ell^\infty(\R), \text{ for any }h\in G,
\end{equation}
for the continuous, linear operator  
\[A'\colon (G,\l^2({\mu_b})) \to \left(\ell^\infty(\R),\|\cdot\|_\infty\right),\  h \mapsto 2\rho_b(H-\mu_b(H)),\] 
with $H(\cdot):=\int_0^\cdot h(v)\d v$.
We want to determine the optimal limiting distribution for estimating the invariant density $\rho_b=\Psi(\P_{t,0})$ in $\ell^\infty(\R)$ in the sense of the convolution theorem 3.11.2 in \cite{vdvw96}. 
Since the distribution of a Gaussian process $\mathfrak G$ in $\ell^\infty(\R)$ is determined by the covariance structure $\text{Cov}(\mathfrak G(x),\mathfrak G(y))$, $x,y\in \mathbb R$, we need to find the Riesz-representer for pointwise evaluations $b^*_x\circ A'\colon G\to\R$, where $b^*_x\colon \ell^\infty(\R)\to\R$, $f\mapsto f(x)$, for any $x\in \R$. 
Stated differently, we need to find the Cram\'er--Rao lower bound for pointwise estimation of $\rho_b(x)$, $x\in\R$. 
Speaking about these one-dimensional targets in $\R$ such as point evaluations $\rho_b(x)$ or linear functionals of the invariant density, we refer to \emph{semiparametric Cram\'er--Rao lower bounds} as the variance of the optimal limiting distribution from the convolution theorem. 
This last quantity is a lower bound for the variance of any limiting distribution of a regular estimator.

Our first step towards this goal is to look at integral functionals which we will use to approximate the pointwise evaluations.
For any continuous, linear functional $b^*\colon \ell^\infty(\R)\to \R$, we can infer from \eqref{differentiability} that 
\begin{equation*}
\sqrt t(b^*(\Psi(\P_{t,h}))-b^*(\Psi(\P_b)))\ \longrightarrow_{t\to\infty}\ b^*(A'h)\quad\text{in }\R,\text{ for all }h\in G.
\end{equation*}
Considering $\Phi_g\colon\ell^\infty(\R)\to \R$, $f\mapsto \int g(x)f(x)\d x$, for a function $g\in C_c^\infty(\R)$, and letting $\Phi_g(\P_{t,h}):=\Phi_g(\rho_{b+t^{-1/2}h})$, this becomes
\[
\sqrt t(\Phi_g(\P_{t,h})-\Phi_g(\P_b))\ \longrightarrow_{t\to\infty}\ \int g(x)(A'h)(x)\d x.
\]
The limit defines a continuous, linear map $\kappa\colon(G,\l^2(\mu_b))\to\R$ with representation
\begin{align}\label{rieszrepresenterableitungintegralfunktional}\nonumber
\kappa(h)&=\ \int g(x)(A'h)(x)\d x \ =\ \int 2g(x)(H(x)-\mu_b(H))\rho_b(x)\d x\\
&=\ 2\langle g_c,H_c\rangle_{\mu_b}\ =\ 2\langle \L_b\L_b^{-1}g_c,H_c\rangle_{\mu_b} \ =\ -  \langle \partial \L_b^{-1}g_c , h\rangle_{\mu_b}.
\end{align}
Here and throughout the sequel, $\L_b$ denotes the generator of the diffusion process $\X$ with drift coefficient $b$, i.e.,
 $\L_b f=b\partial f + \partial^2 f/2$, for any $f\in C_c^\infty(\R)$, and $f_c:=f-\mu_b(f)$ denotes the centered version of $f$, for any function $f\in \l^1(\mu_b)$.
Note that $g_c\in\text{Rg}(\L_b)$ due to the following lemma whose proof is deferred to Section \ref{appeff}.


\begin{lemma}\label{lemma2} 
Let $g\in C_c^\infty(\R)$, and set 
\[
h(z,x):=\frac{\mathds{1}\{z\geq  x\}-F_b(z)}{\rho_b(z)},\quad z,x\in\R.
\]
Then, $g_c$ is contained in the image of the generator $\L_b$, and 
\[\L_b^{-1}(g_c)\ =\ \T(z)\ :=\ \int_{0}^z \int 2g(x)\rho_b(x)h(u,x)\d x\d u.\]
In particular,
\[
\iint g(x)H(x,y)g(y)\d y \d x\ =\  \|\partial \L_b^{-1}(g_c)\|_{\l^2(\mu_b)}^2,
\]
where $H(x,y):=\E[\H(x)\H(y)]$, $x,y\in\R,$ for the Gaussian process $\H$ fulfilling \eqref{cov:H}.
\end{lemma} 
We conclude by means of Theorem 3.11.2 in \cite{vdvw96} that the \emph{Cram\'er--Rao lower bound for estimation of $\Phi_g(\P_b)$} is given by
$\|\partial \L_b^{-1}g_c\|^2_{\l^2(\mu_b)}$.
Using an approximation procedure, it then can be shown that the \emph{Cram\'er--Rao lower bound for pointwise estimation of $\rho_b(y)$} is defined via
\begin{equation}\label{CRLBpointwise}
\CR(y):=\| 2\rho_b(y)h(\cdot,y)\|^2_{\l^2(\mu_b)},\quad \text{for any }y\in\R.
\end{equation}
For details, see Proposition \ref{prop:crpoint} in Section \ref{appeff}.
The same arguments apply to estimation of linear combinations $u\rho_b(x) + v\rho_b(y), u,v,x,y \in \R$, and the corresponding Cram\'er--Rao bound reads $\| 2u\rho_b(x)h(\cdot,x) + 2v\rho_b(y)h(\cdot,y)\|^2_{\l^2(\mu_b)}$. 
It follows that the covariance of the optimal Gaussian process in the convolution theorem is given as
\begin{equation}\label{optimalcovariance}
\CR(x,y):= 4\rho_b(x)\rho_b(y)\int h(z,x)h(z,y)\rho_b(z)\d z, \quad x,y\in\R.
\end{equation}

\subsection{Semiparametric efficiency of the kernel density estimator}
Having characterised the optimal limit distribution in the previous section, it is natural to ask in a next step for an efficient estimator of linear functionals of the invariant density such as pointwise estimation, functionals of the form  $\Phi_g(\P_b):=\mu_b(g)=\int g d\mu_b$ or, even more, for estimation of $\rho_b$ in $\ell^\infty(\R)$.

\begin{definition}\label{def:eff}
 An estimator $\hat\rho_t$ of the invariant density is called \emph{asymptotically efficient} in $\ell^\infty(\R)$ if the estimator is regular, i.e., 
\[\sqrt t \left(\hat \rho_t -  \Psi(\P_{t,h})\right)\ \stackrel{\P_{t,h}}{\Longrightarrow}\ \mathscr L,\quad \text{ as } t\to\infty, \text{ for any } h\in G,\]
for a fixed, tight Borel probability measure $\mathscr L$ in $\ell^\infty(\R)$, 
and if $\mathscr L$ is the law of a centered Gaussian process $\mathfrak G$ with covariance structure $\E\left[\mathfrak G(x)\mathfrak G(y)\right]=\CR(x,y)$ as specified in \eqref{optimalcovariance}, i.e, $\hat\rho_t$ achieves the optimal limiting distribution.
\end{definition}

Given an asymptotically efficient estimator $\hat \rho_t$ in $\ell^\infty(\R)$ and any bounded, linear functional $b^{*}\colon\ell^\infty(\R)\to\R$, efficiency of the estimator $b^{*}(\hat \rho_t)$ for estimation of $b^{*}(\rho_b)$ then immediately follows.
Our next result shows that estimation via $\rho_{t,K}(h)$ for the universal bandwidth choice $h\sim t^{-1/2}$ is suitable for the job. 
Its proof is given in Section \ref{appeff}.
 
 
\begin{theorem}\label{theo:eff}
The invariant density estimator $\rho_{t,K}(t^{-1/2})$ defined according to \eqref{est:dens} is an asymptotically efficient estimator in $\ell^\infty(\R)$.
\end{theorem}

\begin{remark} 
\begin{enumerate}
\item[(a)]
From the proof of Theorem \ref{theo:eff}, it can be inferred that the local time estimator $t^{-1}L^\bullet_t(X)$ is an asymptotically efficient estimator, as well.
\item[(b)]
In terms of earlier research on efficient estimation of the density as a function in $\ell^\infty(\R)$, we shall mention \cite{kutoyants98} and \cite{negri01}. 
The works deal with the efficiency of the local time estimator $t^{-1}L^\bullet_t(X)$ in the LAM sense for certain classes of loss functions. 
Subject of \cite[Section 8]{kutoyants98} are $L^2(\nu)$ risks for some finite measure $\nu$ on $\R$ of the form $t\E_b\int |\tilde \rho_t(x)-\rho_b(x)|^2 \nu(\d x)$, for estimators $\tilde\rho_t$ of $\rho_b$, whereas \citeauthor{negri01} complements this work for risks of the form $\E_b\left[g\left(\sqrt t \|\tilde\rho_t-\rho_b\|_\infty\right)\right]$ for a class of bounded, positive functions $g$. 
Of course, the distribution appearing in the lower bound corresponds to the optimal distribution in the sense of the convolution theorem that we derived. 
The derivation in \cite{kutoyants98} of the lower bound is based on the van Trees inequality as established in \cite{gilllevit95} as an alternative to the classical approach relying on H\'ajek--Le Cam theory. 
On the other hand, \citeauthor{negri01}'s method follows \cite{millar83} and makes use of the idea of convergence of experiments originally provided by Le Cam. 
The optimal distribution in the sense of a convolution theorem is not shown neither is any asymptotic efficiency result in $\ell^\infty(\R)$ for the kernel density estimator. 
\end{enumerate}
\end{remark}

\section{Minimax optimal adaptive drift estimation wrt $\sup$-norm risk}\label{sec:adapt_einfach}
We now turn to the original question of estimating the drift coefficient in a completely data-driven way. 
The aim of this section is to suggest a scheme for rate-optimal choice of the bandwidth $h$, based on a continuous record of observations $X^t\equiv (X_s)_{0\leq s\leq t}$ of a diffusion as introduced in Definition \ref{def:B}, optimality considered in terms of the $\sup$-norm risk.
Since we stick to the continuous framework, our previous concentration results are directly applicable, allowing, e.g., for the straightforward derivation of upper bounds on the variance of the estimator $\overline\rho_{t,K}(h)$ of the order
\begin{equation}\label{def:ovsigma}
\overline\sigma^2(h,t):=\frac{(\log(t/h))^3}{t}+\frac{\log(t/h)}{th}.
\end{equation}
Standard arguments provide for any $b\in\Sigma(\beta,\LL)$ bounds on the associated bias of order $B(h)\lesssim h^\beta$.
In case of known smoothness $\beta$, one can then easily derive the optimal bandwidth choice $h^*_t$ by balancing the components of the bias-variance decomposition 
\[\sup_{b\in\Sigma(\beta,\LL)}\E_b\left[\Big\|\overline\rho_{t,K}(h)-\frac{\rho_b'}{2}\Big\|_\infty\right] \leq B(h) + \K\overline\sigma(h,t),\]
resulting in $h_t^\ast\sim(\log t/t)^{\frac{1}{2\beta + 1}}$.
In order to remove the (typically unknown) order of smoothness $\beta$ from the bandwidth choice, we need to find a data-driven substitute for the upper bound on the bias in the balancing process. 
Heuristically, this is the idea behind the Lepski-type selection procedure suggested in \eqref{est:band0} below.

\paragraph{1.)}
Specify the discrete grid of candidate bandwidths
\begin{equation}\label{def:cand}
\HH \equiv\HH_t\ :=\ \left\{h_k=\eta^{-k}\colon\ k\in \N,\ \eta^{-k}>\frac{(\log t)^{2}}{t}\right\},\quad \eta>1\text{ arbitrary},
\end{equation}
and define $\overline\eta_1:=24 \tilde C_2\|K\|_{L^2(\lebesgue)}\bdg\e\sqrt v$, $\overline\eta_2:=12\bdg\|K\|_{L^2(\lebesgue)}$ and
\begin{equation}\label{tilm}
\tilde M = \tilde M_t\ :=\ C\|\rho_{t,K}(t^{-1/2})\|_\infty,\qquad \text{ for } C=C(K)\ :=\ 20\e^2\left(4\overline\eta_1+2\overline\eta_2\right)^2.
\end{equation}
\paragraph{2.)}
Set
\begin{equation}\label{est:band0}
\tri h_t\ :=\ \max\Bigg\{h\in\HH\colon \|\overline\rho_{t,K}(h)-\overline\rho_{t,K}(g)\|_\infty\leq \sqrt{\tilde M}\overline\sigma(g,t)
\,\,\forall g<h,\ g\in\HH\Bigg\}.
\end{equation}

The constants involved in the definition of $\overline\eta_1$ and $\overline\eta_2$ are explained in Remark \ref{rem:const} in Section \ref{app:a1}.
For the proposed data-driven scheme for bandwidth choice, we obtain the subsequent

\begin{theorem}\label{theo:est}
For $b\in\Sigma(\beta,\LL)$ as introduced in Definition \ref{def:Sigma}, consider the SDE \eqref{SDE}.
Given some kernel $K$ satisfying \eqref{kernel}, define the estimators $b_{t,K}(\tri h_t)$ and $\overline\rho_{t,K}(\tri h_t)$ according to \eqref{est:drift}, \eqref{est:derivative} and \eqref{est:band0}.
Then, for any $0<\beta+1\leq \alpha$,
\begin{align}\nonumber
\sup_{b\in\Sigma(\beta,\LL)}\E_b\left[\Big\|\overline\rho_{t,K}(\tri h_t)-\frac{\rho_b'}{2}\Big\|_\infty\right] &\lesssim\ \left(\frac{\log t}{t}\right)^{\frac{\beta}{2\beta+1}},\\\label{21}
\sup_{b\in\Sigma(\beta,\LL)}\E_b\left[\left\|(b_{t,K}(\tri h_t)-b)\cdot \rho_b^2\right\|_\infty\right]&\lesssim\ \left(\frac{\log t}{t}\right)^{\frac{\beta}{2\beta+1}}.
\end{align}
\end{theorem}
The suggested estimators for $\rho'_b$ and $b$, respectively, are rate optimal as the following lower bounds imply.

\begin{theorem}\label{thm:lowerboundableitung} 
Let $\beta, \mathcal L,\CC, A, \gamma \in (0,\infty)$, and assume that $\Sigma(\beta,\mathcal L/2,\CC/2,A,\gamma)\neq\emptyset$.
Then,
\begin{align}\label{lowerboundableitung}
\liminf_{t \to \infty} \inf_{\tilde{\partial\rho}_t}\sup_{b\in\Sigma(\beta,\mathcal L)}\E_b\left[\left(\frac{\log t}{t}\right)^{-\frac{\beta}{2\beta +1}}\|\tilde{\partial\rho}_t - \rho'_b\|_\infty\right]&>\ 0,\\\label{lowerbounddrift}
\liminf_{t\to\infty}\inf_{\tilde b}\sup_{b\in\Sigma(\beta,\mathcal L)}\E_b\left[\left(\frac{\log t}{t}\right)^{-\frac{\beta}{2\beta + 1}}\left\|(\tilde b - b)\cdot \rho_b^2\right\|_\infty\right]&>\ 0,
\end{align}
where the infimum is taken over all possible estimators $\tilde{\partial \rho}_t$ of $\rho'_b$ and $\tilde b$ of the drift coefficient $b$, respectively.
\end{theorem}

The proof of the preceding theorem is based on classical tools from minimax theory as laid down in \cite{tsy09}.
Precisely, it relies on the Kullback version of Theorem 2.7 in \cite{tsy09}, his main theorem on lower bounds. 
This result slightly reformulated in terms of our problem is stated in Lemma \ref{lem:tsybakov}. 

\section{Simultaneous estimation}\label{sec:sim}
The result presented in the previous section is a classical specification of Lepski's procedure which complements the study of one-dimensional drift estimation in the continuous framework.
However, as will be shown in the sequel, our techniques allow for results which go beyond classical issues such as minimax optimality.
We will adjust the bandwidth selection procedure from Section \ref{sec:adapt_einfach} in such a way that the resulting data-driven bandwidth choice yields an asymptotically efficient estimator of the invariant density and, simultaneously, also gives a drift estimator which achieves the best possible convergence rate in $\sup$-norm loss.
The approach is an adaptation of the method developed by \cite{gini09} to the scalar diffusion set-up.

\paragraph{1.)}
Define the set of candidate bandwidths $\HH=\HH_t$ according to \eqref{def:cand}, and introduce $h_{\min}:=\min\left\{h_k\in\HH\colon k\in\N\right\}$. 
Set 
\[
\wideparen M \ :=\ \wideparen M_t\ :=\ C\|\rho_{t,K}(h_{\min)}\|_\infty,
\]
for the constant $C$ defined in \eqref{tilm}.
\paragraph{2.)}
Set
\begin{align}
\begin{split}\label{est:band}
\hat h_t&:= \max\Bigg\{h\in\HH\colon \|\overline\rho_{t,K}(h)-\overline\rho_{t,K}(g)\|_\infty\leq \sqrt{\wideparen M}\overline\sigma(g,t)\ \forall g<h,\ g\in\HH\\
&\hspace*{10em}\text{ \textbf{and} }
\left\|\rho_{t,K}(h)-\rho_{t,K}(h_{\min})\right\|_\infty\leq\frac{\sqrt h(\log(1/h))^4}{\sqrt t\log t}\Bigg\}.
\end{split}
\end{align}

Our goal is to estimate the invariant density $\rho_b$ and the drift coefficient $b$ via
$\rho_{t,K}(\hat h_t)$ and 
\begin{equation}\label{def_sim_drift_est}
\tilde b_{t,K}(\hat h_t)(x)\ :=\ \frac{\overline\rho_{t,K}(\hat h_t)(x)}{\rho_{t,K}(\hat h_{t})(x) + \sqrt{\frac{\log t}{t}}\exp\left(\sqrt{\log t}\right)},
\end{equation}
respectively, by means of the simultaneous, adaptive bandwidth choice $\hat h_t$.

\begin{theorem}\label{theo:sim}
Grant the assumptions of Theorem \ref{theo:est}. 
Given some kernel $K$ satisfying \eqref{kernel}, define the estimators $\rho_{t,K}(\hat h_t)$, $\overline \rho_{t,K}(\hat h_t)$ and $\tilde b_{t,K}(\hat h_t)$ according to \eqref{est:dens}, \eqref{est:derivative}, \eqref{def_sim_drift_est} and \eqref{est:band}. 
Then,
\begin{equation}\label{sim:1}\tag{$\operatorname{\mathbf{I}}$}
\sqrt t\left(\rho_{t,K}(\hat h_t)-\rho_b\right)\ \stackrel{\P_b}{\Longrightarrow}\ \H,\qquad \text{ as } t\to\infty,
\end{equation}
in $\ell^\infty(\R)$.
Furthermore, for any $0<\beta+1\leq \alpha$,
\begin{align}\nonumber
\sup_{b\in\Sigma(\beta,\LL)}\E_b\left[\Big\|\overline \rho_{t,K}(\hat h_t)-\frac{\rho_b'}{2}\Big\|_\infty\right]&\lesssim\ \left(\frac{\log t}{t}\right)^{\frac{\beta}{2\beta+1}},\\\label{sim:21}\tag{$\operatorname{\mathbf{II}}$}
\sup_{b\in\Sigma(\beta,\LL)}\E_b\left[\left\|(\tilde b_{t,K}(\hat h_t)-b)\cdot\rho_b^2\right\|_\infty\right]&\lesssim\ \left(\frac{\log t}{t}\right)^{\frac{\beta}{2\beta+1}}.
\end{align}
\end{theorem}

\begin{remark}
It could be argued that this simultaneous procedure is of limited relevance because we have verified that the simple choice $h_t=t^{-1/2}$ is optimal for the kernel density estimator.
In particular, an adaptive bandwidth selection procedure is not required, such that it suffices to apply the data-driven selection procedure stated in Section \ref{sec:adapt_einfach} to obtain an optimal bandwidth for the drift estimator. 
Besides being of theoretical interest, Theorem \ref{theo:sim} still creates added value in relevant aspects. 
The approach that we demonstrate actually could be extended, e.g., to the framework of multivariate diffusions processes. 
It is known that in higher dimensions, the kernel invariant density estimator is also rate-optimal, but the optimal bandwidth depends on the unknown smoothness such that it has to be chosen adaptively as it is certainly the case for the drift estimator.
In this situation, it is very appealing to have one bandwidth selection procedure that works simultaneously both for invariant density and drift estimation. 
Our result thus is pioneering work in this direction.
\end{remark}

\bigskip

\begin{appendix}

\section{Proof of results on Donsker-type theorems and asymptotic efficiency}\label{appeff}
The following result verifies the general conditions for the uniform CLT for diffusion local time given in \cite{vdvvz05} for the class of diffusion processes with $b\in\Sigma(\CC,A,\gamma,1)$.

\begin{lemma}\label{lem:gaussexists}
For any $b\in\Sigma(\CC,A,\gamma,1)$, it holds
\begin{equation}\label{eq:2.7}
\sqrt t\left(\frac{L_t^{\bullet}(X)}{t} -\rho_b\right)\ \stackrel{\P_b}{\Longrightarrow}\ \H,\qquad \text{ as } t\to\infty,
\end{equation}
in $\ell^\infty(\R)$, where $\H$ is a centred Gaussian random map with covariance structure specified in \eqref{cov:H}.
\end{lemma}
\begin{proof}
We verify the conditions of Corollary 2.7 in \cite{vdvvz05} by showing that, for any $b\in\Sigma=\Sigma(\CC,A,\gamma,1)$,
\begin{enumerate}
\item[$\operatorname{(a)}$] $\int F_b^2(x)(1-F_b(x))^2\d s(x)<\infty$ and
\item[$\operatorname{(b)}$] $\lim_{x\to-\infty}\rho^2_b(x)|s(x)|\log\log|s(x)|=0$.
\end{enumerate}
With regard to (a), note first that, for $y>A$,
\[
\frac{1-F_b(y)}{\rho_b(y)}\ =\ \int_y^\infty \exp\left(2\int_y^vb(z)\d z\right)\d v\ \leq\ \int_y^\infty\e^{-2\gamma(v-y)}\d v\ =\ \frac{1}{2\gamma}
\] and, for $y<-A$,
\[
\frac{F_b(y)}{\rho_b(y)}\ =\ \int_{-\infty}^y \exp\left(-2\int_v^yb(z)\d z\right)\d v\ \leq\ \int_{-\infty}^y\e^{2\gamma(v-y)}\d v\ =\ \frac{1}{2\gamma}.
\]
Exploiting the relation $\rho_b(\d x) =(s'(x)m(\R))^{-1}\d x$ between the invariant measure and the scale function as well as the speed measure, respectively, we obtain
$\d s(x) = (\rho_b(x)m(\R))^{-1}\d x$.
Consequently, the above bounds imply that
\begin{align*}
&\int F_b^2(x)(1-F_b(x))^2\d s(x)\\
&\quad\leq\ \frac{1}{m(\R)}\int_\R \frac{F_b^2(x)(1-F_b(x))^2}{\rho_b(x)}\d x\\
&\quad\leq\ \frac{1}{m(\R)}\bigg[\int_{-\infty}^{-A}\frac{F_b^2(y)}{\rho^2_b(y)}\rho_b(y)\d y+ 2A\sup_{x\in[-A,A]}\rho_b^{-1}(x)+ \int_A^\infty \frac{(1-F_b(y))^2}{\rho^2_b(y)}\rho_b(y)\d y\bigg] < \infty.
\end{align*}
In order to verify (b), recall first that the scale function $s$ of $X$ is given by
\[s(x)=\int_0^x \exp\left(-2\int_0^yb(z)\d z\right)\d y=C_b\int_0^x \frac{1}{\rho_b(y)}\d y,\quad x\in\R .\] 
Since, for any $b\in\Sigma$, $b(x)\operatorname{sgn}(x)\leq -\gamma$ whenever $|x|>A$, we obtain for $x<-A$
\begin{align*}
\rho_b(x)|s(x)|&=\ C_b \rho_b(x)\int_{-A}^0 \frac{1}{\rho_b(y)}\d y\ +\ C_b \int_{x}^{-A} \frac{\rho_b(x)}{\rho_b(y)}\d y\\
&\lesssim\ o(1)+\int_x^{-A} \exp\left(-2\left(\int_x^0 b(z)\d z - \int_y^0 b(z)\d z\right)\right)\d y\\
&\lesssim\ o(1) + \int_x^{-A}\exp\left(-2\int_x^y b(z)\d z\right)\d y\\
&\lesssim\ o(1) + \int_x^{-A}\exp\left(-2\gamma(y-x)\right)\d y\\
&\sim\ o(1) + \frac{1}{2\gamma}\left(1-\exp\left(2\gamma(A+x)\right)\right)\ \sim \ o(1) + \frac{1}{2\gamma}\ =\  O(1),
\end{align*}
as $x\to-\infty$. Furthermore, for $x<-A$,
\begin{align*}
\rho_b(x)&=\ C_b^{-1}\exp\left(-2\int_x^0b(y)\d y\right)\ \lesssim \
\exp\left(-2\int_x^{-A}b(y)\d y\right)\\
&\lesssim\ \exp(2\gamma(A+x))\ \lesssim\ \e^{2\gamma x},
\end{align*} and, using the linear growth condition on $b$,
\begin{align*}
|s(x)|&=\ \int_x^0 \exp\left(2\int_y^0 b(z)\d z\right)\d y\ \lesssim\ \int_x^0\exp \left(2\int_y^{-A}b(z)\d z\right)\d y\\
&\lesssim\  \int_x^0 \exp\left(2C\int_y^{-A}(1-z)\d z\right)\d y\ \lesssim\ \int_x^0\exp(C(y^2-2y))\d y
\end{align*}
such that $|s(x)|=O\left(1+|x|\exp(4Cx^2)\right)$ and $\log\log|s(x)|=O(x^2)$ as $x\to-\infty$. 
Finally,
\[\rho_b^2(x)\ |s(x)|\ \log\log|s(x)|\ \lesssim \e^{2\gamma x}\ O(1)\ x^2\ =\ o(1)\quad \text{as } x\to-\infty.\]
Thus, condition (b) of Theorem 2.6 in \cite{vdvvz05} is satisfied.
Consequently, there exists a tight version of the Gaussian process $\H$, and \eqref{eq:2.7} holds true.
\end{proof}

The remainder of this section is devoted to complementing our study of asymptotic efficiency by stating the remaining proofs.
We start with verifying differentiability of the operator $G\ni h\mapsto \rho_{b+h}$.

\begin{lemma}\label{lemmadifferentiability} For any $h\in G$, set $H(\cdot):=\int_0^\cdot h(v)\d v$. 
Then, the operator $h\mapsto \rho_{b+h}$ as a function from $(\ell^\infty(\R)\cap \Lip_{\operatorname{loc}}(\R),\|\cdot\|_\infty)$ to $(\ell^\infty(\R),\|\cdot\|_\infty)$ is Fr\'echet-differentiable at $h=0$ in the sense that
\[\|\rho_{b+h}-\rho_b-2\rho_b(H-\mu_b(H))\|_\infty\ =\ o(\|h\|_\infty).\]
\end{lemma}
\begin{proof}
Let $h\in \ell^\infty(\R)\cap\Lip_{\operatorname{loc}}(\R)$, and denote by $\rho_{b+h}$ the invariant density corresponding to the diffusion process with drift $b+h$. 
Note that, for $\|h\|_\infty$ sufficiently small, $b+h\in \Sigma(\tilde\C, \tilde A, \tilde \gamma,1)$ for some positive constants $\tilde\C, \tilde A,\tilde \gamma$.
Set
\[C_g\ :=\ \int_\R \exp\left(2\int_0^x g(v)\d v\right)\d x,\quad g\in\{b,b+h\},\quad B(\cdot)\ :=\ \int_0^\cdot b(v)\d v.
\]
Then, for any $x\in\R$, 
\begin{align*}
\rho_{b+h}(x)-\rho_b(x)
&=\ C_b^{-1}\e^{2B(x)}\left(\frac{C_b}{C_{b+h}}\ \e^{2H(x)}-1\right)\\
& =\ \rho_b(x)\bigg\{2H(x)+\log\left(\frac{C_b}{C_{b+h}}\right)+\frac{1}{2}\left(2H(x)+\log\left(\frac{C_b}{C_{b+h}}\right)\right)^2\\
&\hspace*{18em}\times \e^{\theta_1(x)\left(\log\left(\frac{C_b}{C_{b+h}}\right)+2H(x)\right)}\bigg\},
\end{align*}
for some $\theta_1(x)\in(0,1)$.
Moreover,
\[
\log\left(\frac{C_b}{C_{b+h}}-1+1\right)
\ =\ \frac{C_b}{C_{b+h}}-1+\frac{1/2}{1+\theta_2\left(C_bC_{b+h}^{-1}-1\right)}\left(\frac{C_b}{C_{b+h}}-1\right)^2,
\]
for some $\theta_2\in(0,1)$.
Next, we will show that
\[
\frac{C_b-C_{b+h}}{C_{b+h}}\ =\ -2\int H(v)\rho_b(v)\d v+o(\|h\|_\infty).
\]
Note that
\begin{align*}
C_b-C_{b+h}&=\ \int\e^{2B(v)}\left(1-\e^{2H(v)}\right)\d v\\
&=\ -\int\e^{2B(v)}\left(2H(v)+\frac{1}{2}\e^{2\theta_3(v)H(v)}4H^2(v)\right)\d v, \quad \theta_3(v)\in(0,1),\\
&=\ -C_b\int 2H(v)\rho_b(v)\d v + o(\|h\|_\infty),
\end{align*}
where we have used $|H(v)|\leq |v|\|h\|_\infty$, $v\in\R$, as well as the fact $\int\e^{2\int_0^v b(x) + 2|h(x)|\d x}|v|^2\d v=O(1)$. 
We conclude
\begin{align*}
\frac{C_b-C_{b+h}}{C_{b+h}}&=\ \frac{-2C_b\mu_b(H)+o(\|h\|_\infty)}{(C_{b+h}-C_b)+C_b}\\ 
&=\ -2\mu_b(H)+\frac{2\mu_b(H)(C_{b+h}-C_b)}{(C_{b+h}-C_b)+C_b}+\frac{o(\|h\|_\infty)}{C_{b+h}}\\
&=\ -2\mu_b(H)+\frac{o(1)(C_{b+h}-C_b)}{o(1)+C_b}+\frac{o(\|h\|_\infty)}{C_{b+h}}\\
&=\ -2\mu_b(H)+\frac{o(1)O(\|h\|_\infty)+o(\|h\|_\infty)}{o(1)+C_b}+o(\|h\|_\infty)\\
&=\ -2\mu_b(H)+o(\|h\|_\infty).
\end{align*}
Consequently, $\left(\frac{C_b-C_{b+h}}{C_{b+h}}\right)^2=o(\|h\|_\infty)$, and it follows
\[
\log\left(\frac{C_b}{C_{b+h}}\right)= -2\mu_b(H)+o(\|h\|_\infty)+\frac{1}{2}\ O(1)o(\|h\|_\infty)= -2\mu_b(H)+o(\|h\|_\infty).
\]
Taking everything into consideration,
\begin{align*}
\rho_{b+h}(x)-\rho_b(x)&=\ \rho_b(x)\bigg\{2(H(x)-\mu_b(H))+o(\|h\|_\infty)\\
&\qquad+\left(2(H(x)-\mu_b(H))+o(\|h\|_\infty)\right)^2\ \frac{1}{2}\e^{2\theta_1(x)H(x)}\e^{\theta_1(x)(-2\mu_b(H)+o(\|h\|_\infty))}\bigg\},
\end{align*}
and thus
\begin{align*}
&\left\|\rho_{b+h}-\rho_b-2\rho_b(H-\mu_b(H))\right\|_\infty\\
&\quad\leq\ o(\|h\|_\infty)+\left\|x\,\mapsto \,\rho_b(x)\left\{16\|h\|_\infty^2x^2+o(\|h\|_\infty)\right\}\e^{2\|h\|_\infty|x|}O(1)\right\|_\infty= o(\|h\|_\infty),
\end{align*}
using that $\sup_{x\in\R}\rho_b(x)\e^{2\|h\|_\infty|x|}(1+x^2)\ =\ O(1)$.
\end{proof}

We proceed with verifying the result on the image of the generator $\L_b$ and the expression of $\|\partial \L_b^{-1}(g_c)\|_{\l^2(\mu_b)}$ in terms of $H(x,y)=\E[\H(x)\H(y)]$. 
Recall that $\H$ denotes the centred Gaussian process with covariance structure specified by \eqref{cov:H}.

\begin{proof}[Proof of Lemma \ref{lemma2}]
Rewriting $H(x,y)$ as 
\begin{equation}\label{eq:Hxy}
H(x,y)= 4\rho_b(x)\rho_b(y)\int \Bigg[(\mathds{1}\{[x,\infty)\}(z)-F_b(z))\ (\mathds{1}\{[y,\infty)\}(z)-F_b(z))\Bigg]\rho_b^{-1}(z)\d z
\end{equation}
yields
\begin{align*}
\iint g(x)H(x,y)g(y)\d y \d x&=\ 4\iint g(x)g(y) \rho_b(x)\rho_b(y) \int h(z,x)h(z,y)\rho_b(z)\d z\d x\d y\\
&=\ \int \left[\int 2g(x)h(z,x)\rho_b(x)dx\right]^2\rho_b(z)\d z\\
&=\ \big\|\int 2g(x)\rho_b(x)h(\cdot,x)\d x\big\|_{\l^2(\mu_b)}^2\ =\ \big\|\frac{\d}{\d z}\ \T(z)\ \big\|_{\l^2(\mu_b)}^2,
\end{align*} 
where
\begin{align*}
\T(z)&=\ \int_{0}^z \int 2g(x)\rho_b(x)h(u,x)\d x \d u\\
&=\ \int_{0}^z  \int_{-\infty}^u 2g(x)\rho_b(x)\frac{1-F_b(u)}{\rho_b(u)}\d x\d u
 - \int_{0}^z \int_u^\infty 2g(x)\rho_b(x)\d x \frac{F_b(u)}{\rho_b(u)}\d u.
\end{align*}
Straightforward calculus gives
\begin{align*}
\T'(z)&=\ -\frac{F_b(z)}{\rho_b(z)}\int_z^\infty 2g(x)\rho_b(x)\d x + \frac{1-F_b(z)}{\rho_b(z)}\int_{-\infty}^z 2g(x)\rho_b(x)\d x,\\
\T''(z)&=\ 2 g(z) - \int2g(x)\rho_b(x)\d x- 2b(z)\T'(z).
\end{align*}
One can show that $\T$ and its derivatives satisfy an at most linear growth condition, and it is possible to approximate $\T$ by a sequence of functions $\T_n$ in $C_c^\infty$ such that $\lim_{n\to \infty} \|\partial^k\T_n-\partial^k\T\|_{\l^4(\mu_b)}=0,$ $k=0,1,2$. 
In particular, the at most linear growth condition on $b$ implies that $\T_n$ converges to $\T$ in $\l^2(\mu_b)$ and $\lim_{n\to\infty}\L_b(\T_n)=g-\mu_b(g)$ in $\l^2(\mu_b)$. 
Since $\L_b $ is a closed operator in $\l^2(\mu_b)$, we can conclude that $\T\in\mathcal D(\L_b )$ and $\L_b(\T)=g-\mu_b(g)$.
\end{proof}

In Section \ref{subsec:crinv}, the Cram\'er--Rao lower bound for estimating $\Phi_g(\P_b)=\int g(x)\rho_b(x)\d x=\int g\d\mu_b$ is identified as $\|\partial \L_b^{-1}g_c\|_{\l^2(\mu_b)}^2$.
The following result establishes the corresponding result for pointwise estimation of the invariant density.

\begin{proposition}\label{prop:crpoint}
The Cram\'er--Rao lower bound for pointwise estimation of $\rho_b(y)$, $y\in\R$ fixed, is defined via \eqref{CRLBpointwise}.
\end{proposition}
\begin{proof}
Let $v\colon \R\to \R_+$ be a smooth, symmetric function with support $\supp(v)\subseteq [-1,1$] and $\int v(z) \d z=1$, and define $g^y_{\epsilon}:=\epsilon^{-1}v\left((\cdot-y)\epsilon^{-1}\right)$, for any $y\in\R$. 
Denote by 
\[\CR(y,\epsilon)\ :=\ \Big\|\frac{\mathrm{d}}{\mathrm d z}L_b^{-1}(g^y_\epsilon-\mu_b(g^y_\epsilon))(z)\Big\|_{\l^2(\mu_b)}^2
\] the Cram\'er--Rao lower bound for estimation of $\int g^y_\epsilon d\mu_b.$ Further, note that
 \[\lim_{\epsilon\downarrow 0 } \int g^y_\epsilon d\mu_b=\rho_b(y)=b^*_y(\Psi(\P_b)), \text{ for any } y\in\R,\] where the pointwise evaluation $b^*_y:\ell^\infty(\R)\longrightarrow \R, f\mapsto f(y)$, is an element of the dual of $\ell^\infty(\R)$. We are interested in the Cram\'er--Rao lower bound for pointwise estimation of $\rho_b(y)=b^*_y(\Psi(\P_b))$, $y\in\R$. This bound is given by the squared $\l^2(\mu_b)$-norm of the Riesz representer of $b^*_yA'$. Since $A'h$ is a continuous function, we have, for any $y\in\R$,
\begin{eqnarray*}
b^*_yA'(h)\ =\ \lim_{\epsilon\downarrow 0}\int g^y_\epsilon (x)A'h(x)\d x\ =\ \lim_{\epsilon\downarrow0}-\langle\partial\L^{-1}_b(g^y_\epsilon-\mu_b(g^y_\epsilon)),\ h\rangle_{\mu_b}
\end{eqnarray*}
due to \eqref{rieszrepresenterableitungintegralfunktional}.
We proceed with proving that the limit $\lim_{\epsilon\downarrow 0}  \partial L^{-1}_b(g^y_\epsilon - \mu_b(g^y_\epsilon))$ exists in $\l^2(\mu_b)$. 
Fix $y\in\R$. 
For any $z\in\R$, $z\neq y$, we have for $\epsilon>0$ small enough
\begin{eqnarray*}
\partial\L_b^{-1}(g^y_\epsilon-\mu_b(g^y_\epsilon))(z)
&=& 2\cdot\mathbbm{1}\{y\leq z\}\frac{1-F_b(z)}{\rho_b(z)}\int_{-\infty}^z g_\epsilon^y(x)\rho_b(x) \d x\\
&&\hspace*{3em} - 2\cdot\mathbbm 1 \{y> z\}\frac{F_b(z)}{\rho_b(z)}\int_z^\infty g_\epsilon^y(x) \rho_b(x)\d x,
\end{eqnarray*}
due to Lemma \ref{lemma2} and since $\supp( g_\epsilon^y)\subseteq [y-\epsilon,y+\epsilon]$. 
Moreover, as \[ \max\left\{\sup_{z\in\R}\mathbbm{1}\{y\leq z\}\frac{1-F_b(z)}{\rho_b(z)},\ \sup_{z\in\R}\mathbbm 1 \{y> z\}\frac{F_b(z)}{\rho_b(z)}\right\}<\infty,\]
one obtains
\[\lim_{\epsilon\downarrow 0 } \partial\L_b^{-1}(g^y_\epsilon-\mu_b(g^y_\epsilon))(z)
\ =\ 2\cdot \frac{\mathbbm{1}\{z\geq y \}-F_b(z)}{\rho_b(z)}\rho_b(y)\ =\ 2\rho_b(y) h(z,y)
\]
a.e.~and in $\l^2(\mu_b)$.
We conclude that $b^*_yA'(h)=-\langle2\rho_b(y)h(\cdot,y),h \rangle_{\mu_b}$ such that the assertion follows.
\end{proof}

We are now in a position to prove asymptotic efficiency as defined in Definition \ref{def:eff} for the kernel invariant density estimator with the `universal' bandwidth choice $t^{-1/2}$.

\begin{proof}[Proof of Theorem \ref{theo:eff}]
Rewriting the covariance $\E[\H(x),\H(y)]=H(x,y)$ as in \eqref{eq:Hxy}, one immediately sees that the law of $\H$ corresponds to the optimal distribution of the convolution theorem due to \eqref{optimalcovariance}.
It remains to prove regularity of the estimator. 
\paragraph{Regularity of the estimator $\rho_{t,K}(t^{-1/2})(\cdot)$}
Fix $y\in\R$. 
As we have already seen in the proof of Proposition \ref{donskerdensity},
\[\sqrt t\left(\rho_{t,K}(t^{-1/2}) - t^{-1} L^\bullet_t(\X)\right)(y)\ =\ o_{\P_b}(1),\]
(see \eqref{convinprop}).
We proceed by exploiting the martingale structure obtained from Proposition 1.11 in \cite{kut04},
\begin{align*}
&\sqrt t\left(t^{-1}L^y_t(\X) -\rho_b(y)\right)\\
&\quad =\ \frac{2\rho_b(y)}{\sqrt t}\int_{X_0}^{X_t}\frac{\mathbbm 1\{v>y\}-F_b(v)}{\rho_b(v)}\d v -  \frac{2\rho_b(y)}{\sqrt t}\int_0^t \frac{\mathbbm 1\{X_u>y\}-F_b(X_u)}{\rho_b(X_u)}\d W_u\\
&\quad =\ \frac{R(X_t,y)}{\sqrt t} - \frac{R(X_0,y)}{\sqrt t} - \frac{2\rho_b(y)}{\sqrt t}\int_0^t \frac{\mathbbm 1 \{X_u>y\}-F_b(X_u)}{\rho_b(X_u)}\d W_u,
\end{align*}
with $R(x,y):=2\rho_b(y)\int_0^x \frac{\mathbbm 1\{v>y\}-F_b(v)}{\rho_b(v)}\d v$.
The process $(X_s)_{s\geq0}$ is stationary under $\P_b$, and therefore
\begin{equation*}
 \frac{R(X_t,y)-R(X_0,y)}{\sqrt t}=o_{\P_b}(1).
\end{equation*}
Let $h\in G$, and fix $a,c\in\R$.
Then,
\begin{align*}
&(a,c) \left(\sqrt t\left(\rho_{t,K}(t^{-1/2})(y)-\rho_b(y)\right),\log\left(\frac{\d \P_{t,h}}{\d \P_b}\right)\right)\tt\\
&\quad=\ o_{\P_b}(1) - a\cdot 2\frac{\rho_b(y)}{\sqrt t}\int_0^t \frac{\mathbbm 1\{X_u>y\}-F_b(X_u)}{\rho_b(X_u)}\d W_u\\
&\hspace*{10em} + \ c\cdot \left( \frac{1}{\sqrt t}\int_0^th(X_s)\d W_s - \frac{1}{2}\int h^2(y)\rho_b(y)\d y\right)\\
&\quad=\ o_{\P_b}(1) +  \frac{1}{\sqrt t}\int_0^t \left(-2a\rho_b(y)k(X_u,y) + ch(X_u)\right)\d W_u - \frac{1}{2}c\int h^2(y)\rho_b(y)\d y\\
&\quad\overset{\P_b}{\Longrightarrow}\ \mathcal N \left(-\frac{c}{2}\int h^2(y)\rho_b(y)\d y, \delta^2\right),
\end{align*}
with $k(x,y):= \left(\mathbbm 1\{x>y\}-F_b(x)\right)\rho_b^{-1}(x)$ and
\begin{align*}
\delta^2&=\ \E_b\left[\left(  ch(X_0) - 2a\rho_b(y)k(X_0,y)\right)^2\right]\\
&=\ \E_b\left[4a^2\rho^2_b(y)k^2(X_0,y) + c^2h^2(X_0) - 4ac\rho_b(y)k(X_0,y)h(X_0)\right]\\
&=\ (a,c) \begin{pmatrix}\E_b\left[4\rho_b^2(y)k^2(X_0,y)\right] & -\E_b\left[2\rho_b(y)k(X_0,y)h(X_0)\right]\\-\E_b\left[2\rho_b(y)k(X_0,y)h(X_0)\right] & \E_b\left[h^2(X_0)\right]  \end{pmatrix} 
\begin{pmatrix}a\\ c\end{pmatrix}.
\end{align*}
The Cram\'er--Wold device then implies that 
\begin{align*}
&\left(\sqrt t\left(\rho_{t,K}(t^{-1/2})(y)-\rho_b(y)\right),\log\left(\frac{\d \P_{t,h}}{\d \P_b}\right)\right)\\
&\hspace*{10em} \overset{\P_b}{\Longrightarrow} \mathcal N \left(\begin{pmatrix}0\\ -\frac{1}{2}\E_b(h^2(X_0))\end{pmatrix},\begin{pmatrix}\E_b(4\rho_b^2(y)k^2(X_0,y)) & -\tau\\  -\tau& \E_b(h^2(X_0)) \end{pmatrix}\right),
\end{align*}
where $\tau:=2\E_b\left[\rho_b(y)k(X_0,y)h(X_0)\right]$.
In turn, Le Cam's Third Lemma yields
\[\sqrt t\left(\rho_{t,K}(t^{-1/2})(y)-\rho_b(y)\right)\ \overset{\P_{t,h}}{\Longrightarrow}\ \mathcal N \left(-\tau, 4\E_b\left[\rho_b^2(y)k^2(X_0,y)\right]\right).\]
Furthermore,
\begin{align*}
&\sqrt t\left(\rho_{t,K}(t^{-1/2})(y) - \rho_{b+t^{-1/2}h}(y)\right)\\
&\quad= \sqrt t\left(\rho_{t,K}(t^{-1/2})(y) - \rho_{b}(y)\right)- \sqrt t\left(b^*_y(\Psi(\P_{t,h}))-b^*_y(\Psi(\P_b))\right)\\
&\quad= \sqrt t\left(\rho_{t,K}(t^{-1/2})(y) - \rho_{b}(y)\right)+\langle 2 \rho_b(y)h(\cdot,y),h\rangle_{\mu_b}+o(1)\\
&\quad= \sqrt t\left(\rho_{t,K}(t^{-1/2})(y) - \rho_{b}(y)\right)+\tau+o(1)\\
&\quad\overset{\P_{t,h}}{\Longrightarrow}\ \mathcal N\left(0, \ 4\rho_b^2(y)\int \left(\mathbbm 1 \{z\geq y\}-F_b(z)\right)^2\rho_b^{-1}(z) \d z\right)\ =\ \mathcal N\left(0,\ \E\left[\H^2(y)\right]\right).
\end{align*} 
We conclude that $\rho_{t,K}(t^{-1/2})(y)$ is a regular and consequently efficient estimator of $\rho_b(y)$ for any $y\in\R$. 

\paragraph{Regularity in $\ell^\infty(\R)$}
In an analogous way, it can be shown that all finite-dimensional marginals of 
\begin{equation}\label{eq:regul}
\sqrt t\left(\rho_{t,K}(t^{-1/2}) - \rho_{b+t^{-1/2}h}\right)
\end{equation} 
weakly converge to those of $\H$ under $\P_{t,h}$, for any $h\in G$. 
Therefore, the estimator $\rho_{t,K}(t^{-1/2})$ is also regular in $\ell^\infty(\R)$ if we can show that the process in \eqref{eq:regul} is asymptotically tight. 
As we have already seen that the limiting distribution is optimal, this then gives efficiency of $\rho_{t,K}(t^{-1/2})$ in $\ell^\infty(\R)$.
We proceed as in \cite[Theorem 11.14]{kosorok2008}. 
Fix $\epsilon >0$. Since $\P_{t,h}$ and $\P_b$ are contiguous, $\d\P_{t,h}/\d\P_b$ is stochastically bounded wrt to both $\P_b$ and $\P_{t,h}$.
Hence, we find a constant $M$ such that 
\[\limsup_{t\to\infty}\P_{t,h}\left(\frac{\d \P_{t,h}}{\d\P_b}>M\right)\ \leq\ \frac{\epsilon}{2}.\]
Furthermore, since $\sqrt t(\rho_{t,K}(t^{-1/2})-\rho_{b})$ is asymptotically tight wrt $\P_b$, there exists a compact set $K\subseteq\ell^\infty(\R)$ such that, for any $\delta>0$,
\[\limsup_{t\to\infty}\P_b\left(\sqrt t\left(\rho_{t,K}(t^{-1/2}) - \rho_{b}\right)\in \left(\ell^\infty(\R)\setminus K^\delta\right)^*\right)\ \leq\ \frac{\epsilon}{2M},\]
$K^\delta$ denoting the $\delta$-enlargement of $K$. 
The superscript $^*$ here stands for the minimal measurable cover wrt to both $\P_b$ and $\P_{t,h}$. 
From these choices, we deduce, for any $\delta > 0$, 
\begin{align*}
&\limsup_{t\to\infty}\P_{t,h}\left(\sqrt t\left(\rho_{t,K}(t^{-1/2}) - \rho_b\right)\in \left(\ell^\infty(\R)\setminus K^\delta\right)^*\right)\\
&\quad=\ \limsup_{t\to\infty}\int\mathbbm{1}\left\{\sqrt t\left(\rho_{t,K}(t^{-1/2}) - \rho_b\right)\in \left(\ell^\infty(\R)\setminus K^\delta\right)^*\right\}\d\P_{t,h}\\
&\quad=\ \limsup_{t\to\infty}\int\mathbbm{1}\left\{\sqrt t\left(\rho_{t,K}(t^{-1/2}) - \rho_b\right)\in \left(\ell^\infty(\R)\setminus K^\delta\right)^*\right\} \frac{\d \P_{t,h}}{\d \P_b} \d \P_b\\
&\quad\leq\ \limsup_{t\to\infty} M \int \mathbbm{1}\left\{\sqrt t\left(\rho_{t,K}(t^{-1/2})-\rho_b\right)\in\left(\ell^\infty(\R)\setminus K^\delta\right)^*\right\} \d \P_b\\
&\hspace*{18em}+\limsup_{t\to\infty}\P_{t,h}\left(\frac{\d \P_{t,h}}{\d\P_b}>M\right)\ \leq\ \epsilon.
\end{align*}
Due to the differentiability property \eqref{differentiability}, we conclude that \eqref{eq:regul} is asymptotically tight wrt $\P_{t,h}$, as well. 
We have thus shown that $\rho_{t,K}(t^{-1/2})$ is an efficient estimator in $\ell^\infty(\R)$.
\end{proof}

\bigskip

\section{Proofs for adaptive drift estimation}
\subsection{Preliminaries}\label{app:a1}
We start with stating and proving a number of auxiliary results required for the investigation of the proposed $\sup$-norm adaptive drift estimation procedure.
Recall the definition of $\Sigma(\beta,\LL)$ in \eqref{def:Sigmabeta}.

\medskip

\begin{lemma}\label{biasandvariance}
For $b\in\Sigma(\beta,\mathcal{L})$, $\beta,\mathcal{L}>0$, and for $\overline \rho_{t,K}$ defined according to \eqref{est:derivative} with some kernel function $K\colon\R\to\R$ satisfying \eqref{kernel}, it holds, for any $h\in \mathcal H$,
\begin{align}\label{eqvar}
\sup_{b\in\Sigma(\beta,\mathcal L)}\E_b\left[\|\overline\rho_{t,K}(h)-\E_b\left[\overline\rho_{t,K}(h)\right]\|^2_\infty\right]
&\leq\ \K^2\overline{\sigma}^2(h,t),\\\label{eq:bias}
\sup_{b\in\Sigma(\beta,\mathcal L)}\|\E_b\left[\overline\rho_{t,K}(h)\right]-\rho_b b\|_\infty
&\leq\ B(h),
\end{align}
where $\K$ denotes some positive constant, $\overline\sigma^2(\cdot,\cdot)$ is defined according to \eqref{def:ovsigma} and 
\[
B(h):= h^\beta \ \frac{\LL}{2[\beta]!}\ \int|K(v)||v|^\beta\d v.\]
\end{lemma}
\begin{proof}
Assertion \eqref{eqvar} follows immediately from Proposition \ref{prop:csi}.
For the bias of $\overline\rho_{t,K}$, classical Taylor arguments imply that (see \cite{gini09})
\[
\sup_{b\in\Sigma(\beta,\mathcal L)}\|\E_b\left[\overline\rho_{t,K}(h)\right]-\rho_b b\|_\infty
\ =\ \frac{1}{2}\sup_{x\in\R} \left|\int_\R K_h(x-y)(\rho_b'(y)-\rho_b'(x))\d y\right|\ \leq\ B(h).
\]
\end{proof}

The next two auxiliary results give conditions which allow to translate upper and lower bounds on the $\sup$-norm risk of estimators of $\rho_b'$ into corresponding bounds on the weighted risk of drift estimators.

\medskip

\begin{lemma}[Weighted upper bounds for drift estimation]\label{lem:decomp}
Given $b\in\Sigma(\beta,\mathcal L)$, consider estimators $\rho_t$ and $\overline \rho_t$ of the invariant density $\rho_b$ and $\rho_b'/2$, respectively, fulfilling the following conditions:
\begin{itemize}
\item[$\operatorname{(E1)}$]	
$\exists C_1>0$ such that, for any $p\geq1$,
\[\sup_{b\in\Sigma(\beta,\mathcal L)}\E_b\left[\|\rho_t-\rho_b\|_\infty^{p}\right]\ \leq \ C_1^pt^{-\frac{p}{2}}\left(1+(\log t)\q +  p\q +p^pt^{-\frac{p}{2}}\right);\]
\item[$\operatorname{(E2)}$]	
$\exists C_2>0$ such that $\sup_{b\in\Sigma(\beta,\mathcal L)}\E_b\left[\|\overline\rho_t\|_\infty^2\right]\ \leq\ C_2t^2$;
\item[$\operatorname{(E3)}$]	
$\exists C_3>0$ such that 
$\sup_{b\in\Sigma(\beta,\mathcal L)}\E_b\left[\|\overline\rho_t-\rho_b'/2\|_\infty\right]\ \leq \ C_3 (\log t/t)^{\beta/(2\beta+1)}$.
\end{itemize}
Then, the drift estimator 
\[\widehat b_t(x)\ :=\ \frac{\overline\rho_t(x)}{\rho_t^\ast(x)},\qquad\text{ with } \rho_t^\ast(x):=\rho_t(x)+\sqrt{\frac{\log t}{t}}\exp\left(\sqrt{\log t}\right),\]
satisfies
\[\sup_{b\in\Sigma(\beta,\mathcal L)}\E_b\left[\|(\hat b_t-b)\ \rho_b^2\|_\infty\right]\ =\ O\left((\log t/t)^{\beta/(2\beta+1)}\right).\]
\end{lemma}
\begin{proof} 
For ease of notation, we refrain in the sequel from carrying the $\sup_{b\in\Sigma(\beta,\mathcal L)}$ along.
All arguments hold for the supremum because constants in the upper bounds do not depend on the specific choice of $b\in\Sigma(\beta,\mathcal L)$.
Introduce the set
\[
	B_t\ :=\ \left\{\sqrt t\|\rho_t-\rho_b\|_\infty\leq \sqrt{\log t}\exp(\sqrt{\log t})\right\},\quad t>\e.
\]
For any $p\geq 1$, Markov's inequality and condition (E1) imply that
\begin{align*}
\P_b(B_t\c)
&\leq\ \E_b\left[\|\rho_{t}-\rho_b\|_\infty^p\right]\cdot \left(\frac{t}{\log t}\right)\q\exp(-p\sqrt{\log t})~\\
&\leq\ C_1^pt^{-\frac{p}{2}}\left(1+(\log t)\q +  p\q + \left(p/\sqrt t\right)^p\right)\left(\frac{t}{\log t}\right)\q\exp\left(-p\sqrt{\log t}\right)\\
&\leq\ C_1^p \left((\log t)^{-\frac{p}{2}}+1+\left(\frac{p}{\log t}\right)\q+\left(\frac{p}{\sqrt{t\log t}}\right)^p\right)\exp\left(-p\sqrt{\log t}\right).
\end{align*}
Specifying $p=8\sqrt{\log t}$, one obtains that, for some positive constant $C$,
\[
	\P_b(B_t\c)\ \leq\ C^{8\sqrt{\log t}}\exp(-8\log t).
\]
Thus, on the event $B_t\c$,
\begin{align*}
&\E_b\left[\|(\hat b_t-b)\ \rho_b^2\|_\infty\ \mathds{1}_{B_t\c}\right]\\
&\quad\leq\ \left(\E_b\left[\|\hat b_t\rho_b^2\|_\infty^2\right]\ \P_b(B_t\c)\right)^{1/2}+\E_b\left[\|b\rho_b^2\|_\infty\cdot \mathds{1}_{B_t\c}\right]\\
&\quad\leq\ \sup_{x\in\R}|\rho_b^2(x)|\left(\E_b\left[\|\overline\rho_t\|_\infty^2\right]\ \frac{t}{\log t}\exp(-2\sqrt{\log t})\ \P_b(B_t\c)\right)^{1/2}+\E_b\left[\|b\rho_b^2\|_\infty\cdot \mathds{1}_{B_t\c}\right]\\
&\quad\leq\ \LL^2\left(\E_b\left[\|\overline\rho_t\|_\infty^2\right]\ \frac{t}{\log t}\exp(-2\sqrt{\log t})C^{8\sqrt{\log t}}\exp(-8\log t)\right)^{1/2} +\ \frac{1}{2}\LL^2\P_b(B_t\c)\\
&\quad=\ O\left(\sqrt{C^{8\sqrt{\log t}}\frac{t^3}{\log t}\exp(-8\log t)}+C^{8\sqrt{\log t}}t^{-8}\right)~=~O(t^{-1}).
\end{align*}
On the other hand, $\E_b\left[\|(\hat b_t-b)\ \rho_b^2\|_\infty\ \mathds{1}_{B_t}\right]\leq A_1+A_2$, for
\[
A_1\ :=\ \E_b\left[\Big\|\Big(\hat b_t-\frac{b\rho_b}{\rho_t^\ast}\Big)\ \rho_b^2\Big\|_\infty\cdot \mathds{1}_{B_t}\right],\qquad
A_2\ :=\ \E_b\left[\Big\|\Big(\frac{b\rho_b}{\rho_t^\ast}-b\Big)\ \rho_b^2\Big\|_\infty\cdot\mathds{1}_{B_t}\right].
\]
Since $\rho_b/\rho_t^\ast\leq 1$ on the event $B_t$, $\operatorname{(E1)}$ and $\operatorname{(E3)}$ imply that
\begin{align*}
A_1&\leq\ \sup_{x\in\R}|\rho_b(x)|\ \E_b\left[\|\overline\rho_t-\rho_b'/2\|_\infty\right]\ =\ O\left(\left(\frac{\log t}{t}\right)^{\beta/(2\beta+1)}\right),\\
A_2&=\ \E_b\left[\bigg\|\frac{\rho_b'\rho_b}{2\rho_t^\ast}(\rho_b-\rho_t^\ast)\cdot \mathds{1}_{B_t}\bigg\|_\infty\right]\ \leq\ 
\frac{\LL}{2}\left\{\E_b\left[\|\rho_b-\rho_{t}\|_\infty\right]+\left(\frac{\log t}{t}\right)^{1/2}\exp\left(\sqrt{\log t}\right)\right\}.
\end{align*}
Summing up,
\[
\E_b\left[\|(\hat b_t-b)\ \rho_b^2\|_\infty\right]\ \leq\ O\left(\frac{1}{t}\right)+O\left(\left(\frac{\log t}{t}\right)^{\frac{\beta}{2\beta+1}}\right)
+O\left(\sqrt{\frac{\log t}{ t}}\right)+O\left(\sqrt{\frac{\log t}{ t}}\e^{\sqrt{\log t}}\right),
\]
and the assertion follows.
\end{proof}

\medskip

\begin{lemma}[Weighted lower bounds for drift estimation]\label{lemma:lowdrift}
Assume that, for some $\psi_t$ fulfilling $\sqrt{\log t/t}=o(\psi_t)$, it holds 
\begin{equation}\label{condition:lowdrift}
\liminf_{t \to \infty} \inf_{\tilde{\partial\rho^2_t}}\sup_{b\in\Sigma(\beta,\mathcal L)}
\E_b\left[\psi_t^{-1}\|\tilde{\partial\rho_t^2} - (\rho_b^2)'\|_\infty\right]\ >\ 0.
\end{equation}
Then, 
\[\liminf_{t\to\infty}\inf_{\tilde b}\sup_{b\in\Sigma(\beta,\mathcal L)}\E_b\left[\psi_t^{-1}\left\|(\tilde b - b)\rho_b^2\right\|_\infty\right]\ >\ 0.\]
The infimum in the preceding inequalities is taken over all estimators $\tilde{\partial \rho_t^2}$ and $\tilde b$ of $(\rho_b^2)'$ and $b$, respectively.
\end{lemma}
\begin{proof}
Given any estimator $\tilde b$ of the drift coefficient $b$, define
\begin{align*}
\bar b(x)&:=\ 	\begin{cases} 	\tilde b(x), &\text{if } |\tilde b(x)|\leq \C(1+|x|)\\ 
								\text{sgn}(b(x)) \C (1+|x|), &\text{otherwise}
				\end{cases}, \qquad x\in\R.
\end{align*}
For any $b\in\Sigma(\beta,\mathcal L)$, it holds $|b(x)|\leq \C (1+|x|)$ for all $x\in\R$. Consequently,
\[
\E_b\left[\left\|(\bar b - b)\ \rho^2_b\right\|_\infty\right]\ \leq\ \E_b\left[\big\|(\tilde b - b)\ \rho^2_b\big\|_\infty\right].\]
It thus suffices to consider the infimum over all estimators $\tilde b$ satisfying 
\begin{equation}\label{estimatoratmostlinear}
|\tilde b(x)|\leq \C (1+|x|),\quad x\in\R.
\end{equation}
In view of the decomposition
\[
\tilde b\rho_{t,K}(t^{-\frac{1}{2}})\rho_b - \frac{1}{2}\rho_b'\rho_b\ =\ \frac{1}{2}\rho_b\left(2\tilde b\rho_b-\rho_b'\right) + \tilde b\rho_b\left(\rho_{t,K}(t^{-\frac{1}{2}})-\rho_b\right),
\]
it holds
\[
\E_b\left[\big\|(\tilde b - b)\ \rho_b^2\big\|_\infty\right]\ =\ \E_b\left[\big\|\frac{1}{2}\rho_b\left(2\tilde b\rho_b-\rho_b'\right)\big\|_\infty\right]\ \geq\ (\I) - (\II),
\]
with
\[
(\I):=\frac{1}{4}\E_b\left[\big\|4\tilde b\rho_{t,K}(t^{-\frac{1}{2}})\rho_b - 2\rho_b'\rho_b\big\|_\infty\right],\qquad
(\II):=\E_b\left[\big\|\tilde b\rho_b\left(\rho_{t,K}(t^{-\frac{1}{2}})-\rho_b\right)\big\|_\infty\right].
\]
Due to \eqref{estimatoratmostlinear}, $\sup_{b\in\Sigma(\beta,\mathcal L)}|\tilde b\rho_b|$ is bounded. 
Moreover, we can infer from Proposition \ref{prop:con_kernel_density_estimator} that there exists a positive constant $C_1$ such that, for all $t$ sufficiently large, 
$(\II)\ \leq \ C_1 \sqrt{\log t/t}$. 
Consequently, noting that $4\tilde b\rho_{t,K}(t^{-\frac{1}{2}})\rho_b$ can be viewed as an estimator of $2\rho'_b\rho_b=(\rho^2_b)'$, 
\begin{align*}
\liminf_{t\to\infty}\inf_{\tilde b}\sup_{b\in\Sigma(\beta,\mathcal L)}\E_b\left[\psi_t^{-1}\left\|(\tilde b - b)\ \rho_b^2\right\|_\infty\right]
&\geq\ \liminf_{t\to\infty}\inf_{\tilde b}\sup_{b\in\Sigma(\beta,\mathcal L)} \psi_t^{-1}((\I)-(\II))\\
&\geq\ \liminf_{t\to\infty}\inf_{\tilde b}\sup_{b\in\Sigma(\beta,\mathcal L)} \left(\psi_t^{-1}(\I) - C_1 \psi_t^{-1}\sqrt{\frac{\log t}{ t}}\right)\\
&>\ 0.
\end{align*}
\end{proof}

Proposition \ref{prop:csi} presented in Section \ref{sec:pre} gives one first result on the concentration behaviour of the estimator $\overline{\rho}_{t,K}(h)$ of $\rho_b'/2$.
It follows from a straightforward application of the developments in \cp. 
Note that, for any bandwidth $h=h_t\in(t^{-1}(\log t)^2,\ (\log t)^{-3})$ and any $u=u_t\in[1,\alpha\log t]$, $\alpha>0$, the function $\phi_{t,h}(u)$ introduced in \eqref{phith} fulfills
\[\phi_{t,h}(u)\ =\ \Upsilon_1\sqrt{\frac{\log(ut/h)}{th}}+
\Upsilon_2\sqrt{\frac{u}{th}}+o\left(\sqrt{\frac{\log(ut/h)}{th}}\right).\]
For the construction of an adaptive estimation procedure which yields rate-optimal drift estimators, we need to specify the constants $\Upsilon_1$ and $\Upsilon_2$.
The subsequent Lemma provides the corresponding result. Its proof relies on a modification of the proof of Theorem 18 in \cp.

\medskip

\begin{lemma}[Tail bounds with explicit constants]\label{lem:erg_prop:csi}
Grant the assumptions of Proposition \ref{prop:csi}, and define the estimator $\overline\rho_{t,K}$ according to \eqref{est:derivative}.
Then, for any $h=h_t\in \left(t^{-1}(\log t)^2, (\log t)^{-3}\right)$, $1\leq u=u_t\leq \alpha\log t$, for some $\alpha>0$, and $t$ sufficiently large, it holds
\begin{equation}\label{con_stoch_int_special}
\P_b\left(\sup_{x\in\R}|\overline\rho_{t,K}(h_t)(x)-\E_b\left[\overline \rho_{t,K}(h_t)(x)\right]|>\e\overline\psi_{t,h_t}(u_t)\right)\ \leq\ \e^{-u_t},
\end{equation}
with 
\[
\overline\psi_{t,h}(u):=\sqrt{\|\rho_b\|_\infty}\left\{2\overline\eta_1 \sqrt{\frac{\log(ut/h)}{th}}+\overline\eta_2\sqrt{\frac{u}{th}}\right\},
\]
for $\overline\eta_1=24 \tilde C_2\|K\|_{L^2(\lebesgue)}\bdg\e\sqrt v$ and $\overline\eta_2=12\bdg\|K\|_{L^2(\lebesgue)}$.
\end{lemma}
\begin{proof}
Let us first prove that there exist constants $\co_1,\tilde\co_0$ such that, for any $u,p\geq 1$, $h\in (0,1)$, $t\geq 1$,
\begin{align}
\begin{split}\label{lem:con_stoch_int}
\left(\E_b\left[\|\overline{\rho}_{t,K}(h)-\E_b\left[\overline{\rho}_{t,K}(h)\right]\|_\infty^p\right]\right)\p& \leq\ \tilde \phi_{t,h}(p),\\
\P_b\left(\sup_{x\in\R}|\overline\rho_{t,K}(h)(x)-\E_b\overline \rho_{t,K}(h)(x)|>\e\tilde\phi_{t,h}(u)\right)& \leq\ \e^{-u},
\end{split}
\end{align}
for 
\begin{align*}
\tilde\phi_h(u)\ &:=\ \co_1 \Bigg\{ 
\frac{1}{\sqrt t}\Big\{\left(\log \left(\frac{ut}{h}\right)\right)^{3/2}+\left(\log\left(\frac{ut}{h}\right)\right)^{1/2}+u^{3/2}\Big\}+\frac{u}{th}+  \frac{1}{h} \exp\left(-\tilde\co_0 t \right)\\
&\qquad+ \frac{1}{t^{3/4}\sqrt h}\log\left(\frac{ut}{h}\right)+\frac{1}{\sqrt{th}}  \left\{\left(\frac{\log t}{t}\right)^{1/4}\sqrt u + \frac{u}{t^{1/4}}    \right\}\Bigg\}\\
&\qquad + \,\,\overline\eta_1\sqrt{\|\rho_b\|_\infty}\frac{1}{\sqrt{th}}\left(\log\left(\frac{ut}{h}\right)\right)^{1/2} + \overline\eta_2 \sqrt{\|\rho_b\|_\infty}\left(\frac{u}{th}\right)^{1/2}.
\end{align*}
This statement parallels assertion \eqref{con_stoch_int} from Proposition \ref{prop:csi}.
It remains however to identify the constants preceding the terms $(th)^{-1/2}\left(\log(ut/h)\right)^{1/2}$ and $(u/th)^{1/2}$ as $\sqrt{\|\rho_b\|_\infty}\overline\eta_1$ and $\sqrt{\|\rho_b\|_\infty}\overline\eta_2$, respectively.
This requires to look into the details of the proof of Theorem 5 in \cp. 
In our situation, $\mathcal F$ is definded as in \eqref{functionclass}. 
First, note that the term  $\frac{1}{\sqrt{th}}\left(\log\left(\frac{ut}{h}\right)\right)^{1/2}$ comes from the analysis of the martingale $(th)^{-1}\int_0^t K\left(\frac{y-X_s}{h}\right)\d W_s$. 
We repeat the arguments from the proof of Theorem 5 in \cp, here, to discover the required constant. 
The Burkholder--Davis--Gundy inequality as stated in Proposition 4.2 in \cite{bayo82} and the occupation times formula yield, for any $p\geq2$, $f\in\FF$, 
\begin{align*}
&\left(\E_b\left[\left|\frac{1}{\sqrt t}\int_0^tf(X_s)\d W_s\right|^p\right]\right)^{\frac{1}{p}}\\
&\quad\leq\ \bdg \sqrt p\left(\E_b\left[\left(\frac{1}{t}\int_0^tf^2(X_s)\d s\right)^{\frac{p}{2}}\right]\right)^{\frac{1}{p}}\\
&\quad\leq\ \bdg \sqrt p\|f\|_{L^2(\lebesgue)}\left(\E_b\left[\left(\|t^{-1} L_t^\bullet(X)\|_\infty\right)^{\frac{p}{2}}\right]\right)^{\frac{1}{p}}\\
&\quad\leq\ \bdg \sqrt p\|f\|_{L^2(\lebesgue)}\left\{\left(\E_b\left[\left(\|t^{-1}L_t^\bullet(X)-\rho_b\|_\infty\right)^{\frac{p}{2}}\right]\right)^{\frac{1}{p}} + \sqrt{\|\rho_b\|_\infty}\right\},
\end{align*}
where $\bdg:=\sqrt 2\max\left\{1,\overline{\bdg}\right\}$, for $\overline{\bdg}$ denoting a universal constant from the BDG inequality.
Exploiting the result on centred diffusion local time from Lemma \ref{lem:difflt}, one can deduce that there exists another constant $\co_2>0$ such that, for $\V=\sqrt h\|K\|_{L^2(\lebesgue)}$,
\begin{align*}
&2\sup_{f\in\FF}\left(\E_b\left[\Big|\frac{1}{\sqrt t}\int_0^t f(X_s)\d W_s\Big|^p\right]\right)\p\\
&\quad\leq\ 2\bdg\sqrt p\mathbb V\left(\sqrt{\|\rho_b\|_\infty} + \co_2\left(\left(\frac{\log t}{t}\right)^{1/4} + \left(\frac{p}{t}\right)^{1/4} + \sqrt{\frac{p}{t}}\right)\right)\\
&\quad\leq\ 2\bdg\mathbb V\left(\sqrt{\|\rho_b\|_\infty} + \co_2\left(\frac{\log t}{t}\right)^{1/4}  \right) \sqrt p+4t^{-1/4}\bdg\mathbb V\co_2p.
\end{align*}
Using similar arguments, the latter is verified for $1\leq p<2$.
Furthermore, for any $f,g\in\FF$, it can be shown analogously that
\begin{align*}
\left(\E_b\left[\Big|\frac{1}{\sqrt t}\int_0^t (f-g)(X_s)\d W_s\Big|^p\right]\right)\p
&\leq\ 2\bdg\|f-g\|_{L^2(\lebesgue)}\\
&\quad\times\left\{\sqrt p \left(\sqrt{\|\rho_b\|_\infty} + \co_2\left(\frac{\log t}{t}\right)^{1/4}\right)+\frac{2\co_2p}{t^{1/4}}\right\}.
\end{align*}
Applying the generic chaining method and the localisation procedure as in \cp, one can then verify that
\begin{align*}
&\frac{1}{\sqrt t}\left(\E_b\left[\left\|\int_0^t K\left(\frac{\cdot-X_s}{h}\right)\d W_s\right\|_\infty^p\right]\right)\p\\
&\quad \leq\ \frac{\tilde C_1}{t^{1/4}}\sum_{k=0}^{\infty} \int_0^\infty \log N(u,\mathcal F_k, 4\bdg\co_2\e\|\cdot\|_{\L^2(\lebesgue)})\d u\,\,\e^{-\frac{k}{2}}\\
&\quad\quad+ 
\tilde C_2\sum_{k=0}^{\infty} \int_0^\infty \sqrt{\log N\left(u,\mathcal F_k, 2\bdg\Big(\sqrt{\|\rho_b\|_\infty}+\co_2\left(\frac{\log t}{t}\right)^{1/4}\Big) \e\|\cdot\|_{\L^2(\lebesgue)}\right)}\d u\,\,\e^{-\frac{k}{2}}\\
&\quad\quad+  
12\bdg\V\left(\sqrt{\|\rho_b\|_\infty} + \co_2\left(\frac{\log t}{t}\right)^{1/4}\right) \sqrt p + \frac{24\bdg\V\co_2p}{t^{1/4}}\\
&\quad\leq\ \frac{6\tilde C_1}{t^{1/4}}v\V\ 4\bdg\co_2\e\left(1+\log\left(\frac{\mathbbm A}{\V}\sqrt{\mathcal S + p\Lambda t}\right)\right)+ \mathbf{(A)}+\mathbf{(B)}+\frac{24\bdg\V\co_2p}{t^{1/4}},
\end{align*}
with 
\begin{align*}
\mathbf{(A)}&:=\ 
12\tilde C_2\V \ 2\bdg\left(\sqrt{\|\rho_b\|_\infty}+\co_2\left(\frac{\log t}{t}\right)^{1/4}\right)\e\sqrt{v\log\left(\frac{\mathbbm A}{\V}\sqrt{\mathcal S+p\Lambda t}\right)},\\
\mathbf{(B)}&:=\ 12\bdg\V\left(\sqrt{\|\rho_b\|_\infty}+\co_2\left(\frac{\log t}{t}\right)^{1/4}\right)\sqrt p
\end{align*}
for $\mathbbm A$ and $\mathcal S$ defined as in the proof of Proposition \ref{prop:csi} and some positive constant $\Lambda$.
For some further positive constants $\co_3,\co_4$, we can then upper bound
\begin{align*}
\frac{\mathbf{(A)}}{\sqrt t h} 
&\leq\ \co_3\left(\frac{1}{\sqrt{th}}+\frac{(\log(pt/h))^{3/4}}{t^{3/4}\sqrt h}\right)+\overline\eta_1\sqrt{\|\rho_b\|_\infty\ \frac{\log(pt/h)}{th}},\\
\frac{\mathbf{(B)}}{\sqrt t h} 
&\leq\ \co_4\frac{1}{\sqrt{th}}\left(\frac{\log t}{t}\right)^{1/4}\sqrt p+\overline \eta_2\sqrt{\|\rho_b\|_\infty}\sqrt{\frac{p}{th}}.
\end{align*}
Thus, with regard to the upper bound for $\left(\E_b\left[\|\overline{\rho}_{t,K}(h_t)-\E_b\left[\overline{\rho}_{t,K}(h_t)\right]\|_\infty^p\right]\right)\p$,
the constants preceding the terms $(th)^{-1/2}\sqrt{\log(pt/h)}$ and $\sqrt{p/(th)}$ are identified as $\sqrt{\|\rho_b\|_\infty}\overline\eta_1$ and $\sqrt{\|\rho_b\|_\infty}\overline\eta_2$, respectively. 

Furthermore, we obtained the additional expression $(th)^{-1/2}(\log t/t)^{1/4} \sqrt p$, and assertion \eqref{lem:con_stoch_int} follows.
In order to show \eqref{con_stoch_int_special}, note finally that in case of 
$h_t\in \left((\log t)^2t^{-1},(\log t)^{-3}\right)$ and $1\leq u_t\leq \alpha\log t$, $\alpha>0$, it holds 
\[
\tilde\phi_{t,h_t}(u_t)\ =\ o\left(\frac{1}{\sqrt{th_t}}\sqrt{\log\left(\frac{u_tt}{h_t}\right)}\right) + \overline\eta_1\sqrt{\|\rho_b\|_\infty}\frac{1}{\sqrt{th_t}}\sqrt{\log\left(\frac{u_t t}{h_t}\right)}+\overline\eta_2 \sqrt{\|\rho_b\|_\infty\cdot \frac{u_t}{th_t}}.
\]
Thus, $\tilde\phi_{t,h_t}(u_t)\leq \overline\psi_{t,h_t}(u_t)$ for $t$ sufficiently large.
\end{proof}

\begin{remark}\label{rem:const}
We shortly comment on the constants appearing in the definition of $\overline\eta_1$ and $\overline\eta_2$ in Lemma \ref{lem:erg_prop:csi}.
The constant $\bdg$ is defined as $\bdg=\sqrt 2\max\left\{1,\overline{\bdg}\right\}$, for $\overline{\bdg}$ denoting the universal constant from the Burkholder--Davis--Gundy inequality as stated in Proposition 4.2 in \cite{bayo82}. 
The constant $\tilde C_2$ originates from the application of Proposition \ref{thm:dirk} and can be specified looking into the details of the proof in \cite{dirk15}. 
The constant $v$ is associated with the entropy condition on $\mathcal F$ defined as in $\eqref{functionclass}$ (see Lemma 23 in \cp).
\end{remark}

\subsection{Proof of main results}

This section contains the proof of the upper and lower bound results on minimax optimal drift estimation wrt $\sup$-norm risk presented in Section \ref{sec:adapt_einfach}  and Section \ref{sec:sim}.

\begin{proof}[Proof of Theorem \ref{theo:est}]
In order to prove \eqref{21}, it suffices to investigate the estimator $\overline \rho_{t,K}(\tri h_t)$ since conditions (E1) and (E2) from Lemma \ref{lem:decomp} are satisfied. 
Indeed, the first condition refers to the rate of convergence of the invariant density estimator, and it is fulfilled by the kernel density estimator $\rho_{t,K}$ with bandwidth $t^{-1/2}$ due to Proposition \ref{prop:con_kernel_density_estimator}.
Note that $b\in\Sigma(\beta,\mathcal L)$ in the current paper refers to H\"older continuity of $\rho_b$ with parameter $\beta + 1$ in contrast to the notation in \cp.
With regard to (E2), Lemma \ref{biasandvariance} implies that
\begin{align*}
\sup_{b\in\Sigma(\beta,\LL)}\E_b\left[\|\overline{\rho}_{t,K}(\tri h_t)\|^2_\infty\right]
&\leq \sup_{b\in\Sigma(\beta,\LL)}\Big(4\E_b\left[\|\overline{\rho}_{t,K}(\tri h_t)-\E_b\left[\overline{\rho}_{t,K}(\tri h_t)\right]\|^2_\infty\right]\\
&\hspace*{3em}  + 
4\|\E_b\left[\overline{\rho}_{t,K}(\tri h_t)\right]-b\rho_b\|^2_\infty + 4\|b\rho_b\|_\infty^2\Big)\\
&\lesssim\ \overline\sigma^2(h_{\min},t)+B(1) + 1\ =\ O(1)\quad \text{ as } t\to\infty,
\end{align*}
since $1>\tri h_t\geq h_{\min}=\min\left\{h_k\in\mathcal H\colon k\in \N\right\} > (\log t)^2/t$ and $\|\rho_bb\|_\infty = \|\rho'_b\|_\infty/2\leq \mathcal L$. Hence, (E2) is satisfied.
Consequently, \eqref{21} will follow once we have verified condition (E3) from Lemma \ref{lem:decomp}, i.e., by showing that
\[\sup_{b\in\Sigma(\beta,\LL)} \E_b\left[\|\overline \rho_{t,K}(\tri h_t)-\rho_b'/2\|_\infty\right]
\ =\ O\left(\left(\frac{\log t}{t}\right)^{\frac{\beta}{2\beta+1}}\right).\]
For $C=C(K)$ introduced in \eqref {tilm}, let $M:=C\|\rho_b\|_\infty$, and define $\overline h_{\rho}:=h(\rho_b')$ as
\[
	\overline h_{\rho}\ :=\ \max\left\{h\in\mathcal{H}\colon B(h)\ \leq\ \frac{\sqrt{0.8M}}{4}\ \overline\sigma(h,t)\right\}.
\]
For the given choice of $\overline h_{\rho}$, it holds $B(\overline h_{\rho})\ \simeq\ \sqrt{0.8M}\overline\sigma(\overline h_{\rho},t)/4$, and, since $\mathcal{H}\ni\overline h_{\rho}> (\log t)^2/t$, 
\begin{equation}\label{rate_h_rho}
\overline h_{\rho}^{2\beta+1}\simeq(\log t/t)\quad\text{ and }\quad\overline\sigma(\overline h_{\rho},t)\simeq (\log t/t)^{\frac{\beta}{2\beta+1}}.
\end{equation}
To see this, first note that, for $h_0:=(\log t/t)^{\frac{1}{2\beta +1}}$, there exists some positive constant $L$ such that $B(h_0)\leq L\overline\sigma(h_0,t)$ and $\overline\sigma^2(h_0,t)\eqsim(\log t/t)^{\frac{2\beta}{2\beta +1}}$.
In particular, we have $\overline h_\rho\gtrsim h_0$ which is clear by definition of $\overline h_\rho $ in case that $L\leq \sqrt{0.8M}/4$. 
Otherwise, this follows from the fact that, for any $0<\lambda<1$,
\[B(\lambda h_0)=\lambda^\beta B(h_0)\leq \lambda^\beta L\overline\sigma(h_0,t)\leq \lambda^\beta L\overline\sigma(\lambda h_0,t).\]
For the validity of \eqref{rate_h_rho}, it remains to show that $\overline h_\rho\lesssim h_0$. 
We prove that $\overline h_\rho^{2\beta +1 }h_0^{-(2\beta + 1)}= O(1)$. 
By definition, we have
\begin{equation}\label{comment2}
\overline h_\rho^{2\beta+1}\ \eqsim\ \frac{\overline h_\rho}{t}\ \left(\log\left(\frac{t}{\overline h_\rho}\right)\right)^3 + \frac{1}{t}\log\left(\frac{t}{\overline h_\rho}\right) \
\lesssim\ \frac{\overline h_\rho}{t}\left(\log t\right)^3 + \frac{\log t}{t},
\end{equation}
since, for any $h\in\mathcal H$, $h\geq (\log t)^2/t$.
This implies that
$\overline h_\rho^{2\beta +1}h_0^{-(2\beta + 1)}\lesssim \overline h_\rho (\log t)^2 + 1$.
Again exploiting \eqref{comment2}, we deduce that
\begin{align*}
\overline h_\rho (\log t)^2
&\lesssim\ (\log t)^2\left(\frac{\overline h_\rho}{t}(\log t)^3 + \frac{\log t}{t}\right)^{\frac{1}{2\beta + 1}}\\
&\lesssim\ (\log t)^2t^{-\frac{1}{2\beta+1}}\left(\overline h_\rho(\log t)^3 + \log t\right)^{\frac{1}{2\beta + 1}}\\
&\lesssim\ (\log t)^2t^{-\frac{1}{2\beta + 1}}\left((\log t)^3 + \log t  \right)^{\frac{1}{2\beta + 1}}\ =\ o(1).
\end{align*}
Thus, $\overline h_\rho^{2\beta +1 }h_0^{-(2\beta + 1)}= O(1)$, and we have shown \eqref{rate_h_rho}.

\medskip

\textbf{Case 1:}
We first consider the situation where $\{\tri h_t\geq \overline h_{\rho}\}$.
The definition of $\tri h_t$ according to \eqref{est:band0} and the bias and variance estimates in \eqref{eq:bias} and \eqref{eqvar}, respectively, imply that
\begin{align*}
&\E_b\left[\|\overline\rho_{t,K}(\tri h_t)-\rho_b'/2\|_\infty\ \mathds{1}_{\{\tri h_t\geq \overline h_{\rho}\}\cap\{\tilde M\leq 1.2M\}}\right]\\
&\quad \leq\ \E_b\bigg[\Big(\|\overline\rho_{t,K}(\tri h_t)-\overline\rho_{t,K}(\overline h_{\rho})\|_\infty+
\|\overline\rho_{t,K}(\overline h_{\rho})-\E_b\left[\overline\rho_{t,K}(\overline h_{\rho})\right]\|_\infty\\
&\hspace*{10em} +\|\E_b\left[\overline\rho_{t,K}(\overline h_{\rho})\right]-\rho_b'/2\|_\infty\Big)\ \mathds{1}_{\{\tri h_t\geq \overline h_{\rho}\}\cap\{\tilde M\leq 1.2 M\}}\bigg]\\
&\quad \leq\ \sqrt{1.2 M}\ \overline \sigma(\overline h_{\rho},t)+\K\overline\sigma(\overline h_{\rho},t)+\frac{\sqrt{0.8M}}{4}\ \overline\sigma(\overline h_{\rho},t)\ = \
O(\overline \sigma(\overline h_{\rho},t)).
\end{align*}
Similarly,
\begin{align*}
&\E_b\left[\|\overline\rho_{t,K}(\tri h_t)-\rho_b'/2\|_\infty\ \mathds{1}_{\{\tri h_t\geq \overline h_{\rho}\}\cap\{\tilde M> 1.2M\}}\right]\\
&\quad\leq\ \sum_{h\in\mathcal{H}\colon h\geq \overline h_{\rho}}\E_b\Big[\left(\|\overline\rho_{t,K}(h)-\E_b\left[\overline\rho_{t,K}(h)\right]\|_\infty+B(h)\right)\ \mathds{1}_{\{\tri h_t=h\}\cap\{\tilde M>1.2M\}}\Big]\\
&\quad\lesssim\ \log t \left(\K\overline\sigma(\overline h_{\rho},t)+B(1)\right)\ \sqrt{\E_b\left[\mathds{1}\{\tilde M>1.2M\}\right]}.
\end{align*}
The function $\psi_{t,t^{-1/2}}$ introduced in \eqref{def:psi} obviously fulfills $\psi_{t,t^{-1/2}}(\log t) = o(1).$
Plugging in the definition of $\tilde M$ (see \eqref{tilm}), one thus obtains by means of Proposition \ref{prop:con_kernel_density_estimator}, for $t$ sufficiently large,
\begin{align}
\begin{split}\label{eq:m}
\P_b\left(|\tilde M-C\|\rho_b\|_\infty|>0.2C\|\rho_b\|_\infty\right)
&=\ \P_b\left(\left|\|\rho_{t,K}(t^{-1/2})\|_\infty-\|\rho_b\|_\infty\right| >0.2\|\rho_b\|_\infty\right)\\
&\leq\ \P_b\left(\|\rho_{t,K}(t^{-1/2})-\rho_b\|_\infty >\e\psi_{t,t^{-1/2}}(\log t)\right)\\
&\leq\ \ t^{-1}.
\end{split}
\end{align}
Consequently, we have shown that
\[
\E_b\left[\|\overline\rho_{t,K}(\tri h_t)-\rho_b'/2\|_\infty\ \mathds{1}_{\{\tri h_t\geq \overline h_{\rho}\}}\right]\ =\ O(\overline \sigma(\overline h_\rho,t)).
\]

\medskip

\textbf{Case 2:} It remains to consider the case $\{\tri h_t< \overline h_{\rho}\}$.
Decomposing again as in the proof of Theorem 2 in \cite{gini09}, we have
\begin{align*}
&\E_b\left[\|\overline\rho_{t,K}(\tri h_t)-\rho_b'/2\|_\infty\ \mathds{1}_{\{\tri h_t< \overline h_{\rho}\}\cap\{\tilde M< 0.8M\}}\right]\\
&\quad\leq\ \sum_{h\in\mathcal{H}\colon h< \overline h_{\rho}}\E_b\Big[\left(\|\overline\rho_{t,K}(h)-\E_b\left[\overline\rho_{t,K}(h)\right]\|_\infty+B(h)\right)\ \mathds{1}_{\{\tri h_t=h\}\cap\{\tilde M<0.8M\}}\Big]\\
&\quad\lesssim\ \log t \left(\K\overline\sigma(h_{\min},t)+B(\overline h_{\rho})\right)\ \sqrt{\E_b\left[\mathds{1}\{\tilde M<0.8M\}\right]}\ =\ O(\overline \sigma(\overline h_\rho,t)),
\end{align*}
where we used \eqref{eq:m} for deriving the last inequality.
Furthermore,
\begin{align*}
&\E_b\left[\|\overline\rho_{t,K}(\tri h_t)-\rho_b'/2\|_\infty\ \mathds{1}_{\{\tri h_t< \overline h_{\rho}\}\cap\{\tilde M\geq 0.8M\}}\right]\\
&\quad \leq \sum_{h\in\mathcal{H}\colon h< \overline h_{\rho}}
\E_b\Big[\left(\|\overline\rho_{t,K}(h)-\E_b\left[\overline\rho_{t,K}(h)\right]\|_\infty+\|\E_b\left[\overline\rho_{t,K}(h)\right]-\rho'_b/2\|_\infty\right)\ \mathds{1}_{\{\tri h_t=h\}\cap\{\tilde M\geq0.8M\}}\Big]\\
&\quad\leq \sum_{h\in\mathcal{H}\colon h<\overline h_{\rho}}\K\overline\sigma(h,t)\cdot\sqrt{\P_b\left(\{\tri h_t=h\}\cap\{0.8M\leq \tilde M\}\right)}+ B(\overline h_{\rho}).
\end{align*}
Since the latter summand is of order $O(\overline\sigma(\overline h_{\rho},T))$, we may focus on bounding the first term.
Using again the arguments of \cite{gini09}, the proof boils down to verifying that the term
\[
(\mathbf{I}):=\sum_{g\in\HH\colon g\leq h} \P_b\left(\|\overline\rho_{t,K}(h^+)-\overline\rho_{t,K}(g)\|_\infty>\sqrt{0.8 M}\overline\sigma(g,t)\right),
\]
with $h^+:=\min\{g\in\mathcal{H}:\, g>h\}$, satisfies
\[
\sum_{h\in\HH\colon h<\overline h_{\rho}}\overline\sigma(h,t)\sqrt{(\mathbf{I})}\ =\ O(\overline\sigma(\overline h_{\rho},t)).
\]
Analogously to \cite{gini09}, we start by noting that, for $g<h^+\leq \overline h_\rho$,
\begin{align*}
&\|\overline \rho_{t,K}(h^+)-\overline\rho_{t,K}(g)\|_\infty\\
&\qquad\leq\ \|\overline \rho_{t,K}(h^+)-\E_b\left[\overline\rho_{t,K}(h^+)\right]\|_\infty + \|\overline \rho_{t,K}(g)-\E_b\left[\overline\rho_{t,K}(g)\right]\|_\infty +B(h^+)+B(g)\\
&\qquad\leq\ \|\overline \rho_{t,K}(h^+)-\E_b\left[\overline\rho_{t,K}(h^+)\right]\|_\infty + \|\overline \rho_{t,K}(g)-\E_b\left[\overline\rho_{t,K}(g)\right]\|_\infty
 +\frac{1}{2}\sqrt{0.8M}\overline\sigma(g,t).
\end{align*}
Thus,
\begin{align*}
&\P_b\left(\|\overline\rho_{t,K}(h^+)-\overline\rho_{t,K}(g)\|_\infty>\sqrt{0.8 M}\overline\sigma(g,t)\right)\\
&\qquad\leq\ \P_b\left(\|\overline\rho_{t,K}(h^+)-\E_b\left[\overline \rho_{t,K}(h^+)\right]\|_\infty>\frac{\sqrt{0.8 M}}{4}\overline\sigma(h^+,t)\right)\\
&\hspace*{8em}+ \P_b\left(\|\overline\rho_{t,K}(g)-\E_b\left[\overline\rho_{t,K}(g)\right]\|_\infty>\frac{\sqrt{0.8 M}}{4}\overline\sigma(g,t)\right).
\end{align*} 
We want to apply Lemma \ref{lem:erg_prop:csi} for bounding the last two terms, and for doing so, we verify that
\[\e\overline\psi_{t,g}(\log(1/g))\ \leq\ \frac{\sqrt{0.8 M}}{4}\overline\sigma(g,t).\]
Indeed,
\begin{align*}
\e\overline\psi_{t,g}(\log(1/g))
&=\ \e\sqrt{\|\rho_b\|_\infty}\left\{2\overline\eta_1\frac{1}{\sqrt{tg}}\left(\log\left(\frac{\log(1/g)t}{g}\right)\right)^{1/2} + \overline\eta_2\left(\frac{\log(1/g)}{tg}\right)^{1/2}\right\} \\
&\leq\ \e (4\overline\eta_1  +2\overline\eta_2)\sqrt{\|\rho_b\|_\infty }\left(\frac{\log(t/g)}{tg}\right)^{1/2}  \\
&\leq\ \e (4\overline\eta_1  +2\overline\eta_2)\sqrt{\|\rho_b\|_\infty } \left(\frac{1}{\sqrt t}\left(\log \left(\frac{t}{g}\right)\right)^{3/2}   +   \left(\frac{\log(t/g)}{tg}\right)^{1/2}\right) \\
&= \ \e (4\overline\eta_1  +2\overline\eta_2)\sqrt{\|\rho_b\|_\infty } \overline \sigma(g,t)\ \leq\ \frac{\sqrt{0.8 M}}{4}\overline\sigma(g,t), 
\end{align*}
for $t$ sufficiently large and for $M=20\e^2 (4\overline\eta_1  +2\overline\eta_2)^2 \|\rho_b\|_\infty$ since $h_{\min}\leq g\leq \overline h_\rho$.
Lemma \ref{lem:erg_prop:csi} then implies that, for every $g\leq \overline h_\rho$, $g\in\mathcal{H}$ and $t$ large enough,
\begin{align*}
&\P_b\left(\|\overline \rho_{t,K}(g)-\E_b\left[\overline \rho_{t,K}(g)\right]\|_\infty>\frac{\sqrt{0.8M}}{4}\overline \sigma(g,t)\right)\\
&\qquad\leq\ \P_b\left(\|\overline \rho_{t,K}(g)-\E_b\left[\overline \rho_{t,K}(g)\right]\|_\infty>\e\overline\psi_{t,g}\left(\log\left(\frac{1}{g}\right)\right)\right)\ \leq \ g.
\end{align*}
Consequently,
\begin{align*}
\sum_{g\in\HH\colon g\leq h} 
\P_b\left(\|\overline\rho_{t,K}(h^+)-\overline\rho_{t,K}(g)\|_\infty>\sqrt{0.8 M}\overline\sigma(g,t)\right)
&\leq\ \sum_{g\in\HH\colon g\leq h}(h^+ + g)\\ &\lesssim\ h\log t
\end{align*} 
and
\begin{equation}\label{theo:einfachI}
\sum_{h\in\HH\colon h<\overline h_{\rho}}\overline\sigma(h,t)\sqrt{(\mathbf{I})}\ =\ O\left(t^{-\frac{1}{2}}(\log t)^2\right)\ =\ o(\overline\sigma(\overline h_\rho,t)).
\end{equation}
\end{proof}

We now turn to proving the result on lower bounds for drift estimation wrt $\sup$-norm risk.

\begin{proof}[Proof of Theorem \ref{thm:lowerboundableitung}]
We start with stating a crucial auxiliary result.

\begin{lemma}[Theorem 2.7 in \cite{tsy09}] \label{lem:tsybakov} 
Fix $\CC,A,\gamma,\beta,\mathcal L\in (0,\infty)$, and assume that there exist a finite set $J_t=\{0,\ldots,M_t\}$, $M_t\in\N$, and hypotheses $\{b_j:\, j\in J_t\}\subseteq  \Sigma(\beta, \mathcal L)$ satisfying
\begin{itemize}
\item[$\operatorname{(a)}$]
$\quad
\|\rho'_j - \rho_k'\|_\infty\geq 2\psi_t>0,\quad \text{ for any }j\neq k,\, j,k \in J_t,$ or
\item[$\operatorname{(b)}$]
$\quad
\|(\rho_j^2)' - (\rho^2_k)'\|_\infty\geq 2\psi_t>0,\quad \text{ for any }j\neq k,\, j,k \in J_t,$
\end{itemize}
together with the condition that, for any $j\in J_t$, $\P_{b_{j}}=:\P_j\ll \P_0$, and 
\[
\frac{1}{|J_t|}\sum_{j\in J_t}\KL(\P_j,\P_0) = \frac{1}{|J_t|}\sum_{j\in J_t} \E_j\left[\log\left(\frac{\d\P_j}{\d\P_0}\left(X^t\right)\right)\right]\ \leq\ \alpha \log(|J_t|), 
\]
for some $0<\alpha<1/8$.
Here, $\P_{b_j}$ is the measure of the ergodic diffusion process defined via the SDE \eqref{SDE} with drift coefficient $b_j$ and the corresponding invariant density $\rho_j$, $ j\in J_t$.
Then, in case of $\operatorname{(a)}$ and $\operatorname{(b)}$, respectively, it follows
\begin{align*}
\inf_{\tilde {\partial\rho}_t}\sup_{b\in\Sigma(\beta, \mathcal L)} \E_b\left[\psi_t^{-1}\|\tilde{\partial\rho}_t - \rho'_b\|_\infty\right]&\geq\ c(\alpha)>0,\\
\inf_{\tilde {\partial\rho^2_t}}\sup_{b\in\Sigma(\beta, \mathcal L)} \E_b\left[\psi_t^{-1}\|\tilde{\partial\rho^2}_t - (\rho^2_b)'\|_\infty\right]&\geq\ c(\alpha)>0,
\end{align*}
where the constant $c(\alpha)$ depends only on $\alpha$ and the infimum is taken over all estimators $\tilde{\partial \rho}_t$ of $\rho'_b$ and estimators $\tilde{\partial\rho^2_t}$ of $(\rho^2_b)',$ respectively.  
\end{lemma}
\paragraph{Proof of the lower bound for estimating $\rho_b'$}
The lower bound \eqref{lowerboundableitung} follows by a straightforward application of Lemma \ref{lem:tsybakov}.

\bigskip 

\textbf{\texttt{Step 1:} Construction of the hypotheses.} 
Fix $\beta,\mathcal L,\CC,\gamma,A \in (0,\infty)$, and let $b_0\in\Sigma(\beta,\mathcal L/2,\CC/2, A,\gamma)\subseteq \Sigma(\beta,\mathcal L,\CC, A,\gamma)$. 
In addition, set $J_t:= \{0, \pm 1, \ldots, \pm(\lfloor A (2h_t)^{-1} \rfloor - 1)\}$, $x_j:=2h_t j, j\in J_t$, and 
\[
h_t\ :=\ v\left(\frac{\log t}{t}\right)^{\frac{1}{2\beta + 1}},
\]
$v<1$ some positive constant which will be specified later. 
Let $Q\colon\R\to\R$ be a function that satisfies $Q\in C_c^\infty(\R)$, $\supp(Q)\subseteq [-\frac{1}{2},\frac{1}{2}]$, $Q\in \mathcal H(\beta +1, \frac{1}{2})$, $\left|Q'(0)\right|>0$ and $\int Q(x)\d x =0$. 
The hypotheses will be constructed via additive perturbations of the invariant density $\rho_0:=\rho_{b_0}\in\mathcal H\left(\beta +1,\mathcal L/2\right)$ associated with the drift coefficient $b_0$. 
For this purpose, define
\[ G_0 :\equiv 0,\qquad G_j(x):=\mathcal L h_t^{\beta +1}Q\left(\frac{x-x_j}{h_t}\right),\quad x\in\R,\, j\in J_t\setminus\{0\},\]
and the hypotheses $\rho_j:=\rho_0 + G_j$, $j\in J_t$.
Fix $j\in J_t\setminus\{0\}$, and note that 
\[\partial^k G_j(x)\ =\ \mathcal L h_t^{\beta + 1 - k}\partial^k Q\left(\frac{x-x_j}{h_t}\right),\quad x\in\R,\, k=0,...,\lfloor\beta + 1\rfloor.\]
We immediately deduce that $\|\partial^k G_j\|_\infty \leq \mathcal L/2$ since $h_t\leq 1$ for $t$ sufficiently large and $Q\in\mathcal H\left(\beta+1, \frac{1}{2}\right)$.
Furthermore, for any $x,y\in\R$,
\begin{align*}
\left|\partial^{\lfloor \beta +1 \rfloor}G_j(x)-\partial^{\lfloor \beta +1 \rfloor}G_j(y)\right|
&=\ \mathcal L h_t^{\beta+1-\lfloor \beta +1 \rfloor}\left|\partial^{\lfloor \beta +1 \rfloor}Q\left(\frac{x-x_j}{h_t}\right)-\partial^{\lfloor \beta +1 \rfloor}Q\left(\frac{y-x_j}{h_t}\right)\right|\\
&\leq\ \mathcal L h_t^{\beta + 1-\lfloor \beta +1 \rfloor}\frac{1}{2}\left(\frac{|x-y|}{h_t}\right)^{\beta +1 - \lfloor\beta + 1\rfloor}\ =\ \frac{\mathcal L}{2}|x-y|^{\beta+1-\lfloor\beta+1\rfloor}.
\end{align*}
Thus, $G_j \in \mathcal H(\beta+1,\mathcal L/2)$, and, in particular,
\begin{equation}\label{rhojHoelder}
\rho_j\ =\ \rho_0 + G_j \in \mathcal H(\beta+1,\mathcal L),\quad \text{ for all }j\in J_t.
\end{equation}
Note that, since $\supp(G_j)\subset (-A,A)$ for all $j\in J_t$, we have $G_j(x)=0$ for all $|x|\geq A$. 
Thus, for verifying non-negativity, it suffices to show that $\rho_j(x)\geq 0$ for any $x\in (-A,A)$.
Furthermore, $\inf_{-A\leq x \leq A}\rho_0(x)>0$, and, since 
\[\|G_j\|_\infty\ \leq\ \mathcal L h_t^{\beta + 1} \left\|Q\left(\frac{\cdot-x_j}{h_t}\right)\right\|_\infty\ \leq\ \frac{\mathcal L}{2}h_t^{\beta +1}\ =\ o(1),\]
it holds $\liminf_{t\to\infty}\min_{j\in J_t}\inf_{-A\leq x \leq A}\rho_j (x)>0$. 
For $t$ sufficiently large, we therefore find a constant $c_*$ such that 
\begin{equation}\label{lowerboundrhoj}
\min_{j\in J_t}\inf_{-A\leq x \leq A}\rho_j(x)\ \geq\ c_*\ >\ 0.
\end{equation}
Since, in addition, $\int G_j(x)\d x = 0$, $\rho_j$ is a probability density for each $j\in J_t$.
Let
\[b_j\ :=\ \frac{\rho_j'}{2\rho_j},\quad j\in J_t,\]
and note that $b_j \in \text{Lip}_{\text{loc}}(\R)$. 
Furthermore, $b_j$ can be rewritten as 
\begin{equation}\label{bj}
b_j = \frac{1}{2}\frac{\rho'_0 + G'_j}{\rho_0 + G_j} = b_0 + \frac{1}{2}\left(\frac{\rho'_0 + G'_j}{\rho_0 + G_j}-\frac{\rho'_0}{\rho_0}\right)
= b_0 + \frac{1}{2}\left(\frac{G'_j\rho_0 -  \rho'_0G_j}{\rho_0(\rho_0 + G_j)}\right).
\end{equation}
From this representation, we directly see that $b_0=b_j$ on $\R\setminus(-A,A)$. 
Since $\|G_j\|_\infty + \|G'_j\|_\infty = o(1)$ and due to \eqref{lowerboundrhoj}, we have that 
\[\sup_{-A<x<A}\frac{1}{2}\left|\frac{G'_j\rho_0 -  \rho'_0G_j}{\rho_0(\rho_0 + G_j)}\right|\ \leq\ \frac{\CC}{2},\quad \text{for }t \text{ sufficiently large.}\]
In particular, this implies
\[|b_j(x)|\ \leq\ \frac{\CC}{2}(1+|x|) + \frac{\CC}{2}\ \leq\ \CC(1+|x|), \quad \forall x\in \R,\, t \text{ sufficiently large, } \]
and we can deduce that, for any $j\in J_t$, $b_j\in \Sigma(\CC, A, \gamma,1)$. 
Therefore, each $b_j$ gives rise to an ergodic diffusion process via the SDE $\d X_t=b_j(X_t)\d t + \d W_t$ with invariant density $\rho_j$.
Taking into consideration \eqref{rhojHoelder}, we have thus shown that
\[b_j\in \Sigma(\beta,\mathcal L,\CC, A, \gamma),\quad \forall j\in J_t,\]
for $t$ sufficiently large.

\medskip

\textbf{\texttt{Step 2:} Evaluation of the Kullback--Leibler divergence between the hypotheses.}
From Girsanov's theorem, it can be deduced that 
\[
\KL(\P_j,\P_0)\ =\ \E_j\left[\log\left(\frac{\rho_j(X_0)}{\rho_0(X_0)}\right)\right] + \frac{1}{2}\E_j\left[\int_0^t \left(b_0(X_u)-b_j(X_u)\right)^2\d u\right]\ =:\ (\I_j) + (\II_j).
\]
We first prove boundedness of $(\I_j)$. 
Note that $\rho_j/\rho_0\equiv1$ on $\R\setminus(-A,A)$ and 
\[
\max_{j\in J_t}\sup_{-A\leq x \leq A}\frac{\rho_j(x)}{\rho_0(x)}\ \leq\ \frac{\mathcal L}{\inf_{-A\leq x \leq A }\rho_0(x)},\qquad
\min_{j\in J_t}\inf_{-A\leq x \leq A}\frac{\rho_j(x)}{\rho_0(x)}\ \geq\ \frac{c_*}{\mathcal L}.
\] 
Therefore, $\rho_j/\rho_0$ is bounded away from zero and infinity, uniformly for all $j\in J_t$ and $t$ sufficiently large. 
In particular, $\max_{j\in J_t}|(\I_j)|\ =\ O(1)$.
We now turn to analysing $(\II_j)$. 
From \eqref{bj}, we can deduce, for any $j\in J_t$,
\begin{equation}\label{Zerlegungb-bj}
b_0-b_j\ =\ \frac{b_0G_j}{\rho_0 + G_j} - \frac{1}{2}\frac{G'_j}{\rho_0 + G_j}.
\end{equation}
Since $G_j\equiv 0$ on $\R\setminus (-A,A)$,
\[
\max_{j\in J_t}\left\|\frac{b_0G_j}{\rho_0 + G_j}\right\|_\infty\ \leq\ \max_{j\in J_t} \sup_{-A\leq x \leq A}|b_0(x)| \frac{\|G_j\|_\infty}{c_*}\ =\ O\left(h_t^{\beta + 1}\right),
\]
and, consequently,
\[
\max_{j\in J_t}\E_j\left[\int_0^t \left(\frac{b_0(X_u)G_j(X_u)}{\rho_0(X_u) + G_j(X_u)}\right)^2\d u\right]=O\left(t h_t^{2(\beta + 1)}\right)
= O\left(v^{2\beta + 2}\log t\right).
\]
For the second term on the rhs of \eqref{Zerlegungb-bj}, we calculate
\begin{align*}
&\max_{j\in J_t} \E_j\left[\int_0^t \left(\frac{(G'_j)^{2}(X_u)}{4(\rho_0(X_u)+G_j(X_u))^2}\right) \d u\right] \\
&\quad =\ \max_{j\in J_t}  \frac{t}{4} \int_{-A}^A \frac{\mathcal L^2 h_t^{2\beta}(Q')^2\left(\frac{x-x_j}{h_t}\right)}{(\rho_0(x) + G_j(x))^2}\left(\rho_0(x) + \mathcal L h_t^{\beta +1}Q\left(\frac{x-x_j}{h_t}\right)\right)\d x\\
&\quad \lesssim\ \max_{j\in J_t} c_*^{-2}\Bigg[t \int_{-A}^A h_t^{2\beta}(Q')^2\left(\frac{x-x_j}{h_t}\right)\rho_0(x)\d x\\
&\hspace*{5em}+ \int_{-A}^A(Q')^2\left(\frac{x-x_j}{h_t}\right)\left|Q\left(\frac{x-x_j}{h_t}\right)\right| \d x h_t^{3\beta + 1}\Bigg]\\
&\quad\lesssim\ th_t^{2\beta +1}\mathcal L \int (Q')^2(x)\d x + 2 t A\|Q'\|_\infty^2\|Q\|_\infty h_t^{3\beta +1}\\
&\quad\lesssim\ th_t^{2\beta +1}\ \eqsim\ v^{2\beta + 1} \log t.
\end{align*}
This implies that $\max_{j\in J_t} (\II_j)\ =\ O(v^{2\beta +1}\log t)$ such that
\[\max_{j\in J_t} \KL(\P_j,\P_0)\ =\ O(v^{2\beta +1}\log t).\]

\medskip

\textbf{\texttt{Step 3:} Deducing the lower bound by application of Lemma \ref{lem:tsybakov}.}
For any $j\neq k, j,k \in J_t$, we have
\begin{align*}
\|\rho'_j-\rho'_k\|_\infty &=\ \|G'_j - G'_k\|_\infty = \mathcal Lh_t^{\beta}\sup_{x\in\R}\left|Q'\left(\frac{x-x_j}{h_t}\right) - Q'\left(\frac{x-x_k}{h_t}\right)\right|\\
&\geq \ \mathcal L h_t^\beta |Q'(0)| =\mathcal L|Q'(0)| v^\beta\left(\frac{\log t}{t}\right)^{\frac{\beta}{2\beta + 1}}.
\end{align*}
Here we used that $|x_j-x_k|\geq 2 h_t$ implies that $x_j\notin \supp\left(Q'\left(\frac{\cdot-x_k}{h_t}\right)\right)$, noting that
\[\supp\left(Q'\left(\frac{\cdot-x_k}{h_t}\right)\right)\subseteq (x_k-h_t,x_k + h_t).\]
Furthermore, the number of hypotheses $|J_t|$ satisfies
\[|J_t|\ =\ 2 \left\lfloor\frac{A}{2 h_t}\right\rfloor - 1\ \eqsim\ v^{-1} \left(\frac{ t}{\log t}\right)^{\frac{1}{2\beta + 1}}.\]
Consequently, there is a positive constant $c_1$ such that 
$\log(|J_t|) \geq c_1 \log t$, for all $t$ sufficiently large. 
From the arguments in Step 2, it is clear that $v$ can be chosen small enough such that, for some positive constant $c_2$,
\[\frac{1}{|J_t]}\sum_{j\in J_t} \KL(\P_j,\P_0)\ \leq\ c_2v^{2\beta + 1}\log t\ \leq\ \frac{1}{10}c_1\log t\ \leq\ \frac{1}{10}\log(|J_t|),\] 
for all $t$ sufficiently large.
\eqref{lowerboundableitung} now follows immediately from Lemma \ref{lem:tsybakov}.
 
\paragraph{Proof of the weighted lower bound for drift estimation}
For proving \eqref{lowerbounddrift}, we use Lemma \ref{lemma:lowdrift} and the following 

\begin{proposition}\label{thm:lower_bound_rho'rho} 
Grant the assumptions of Theorem \ref{thm:lowerboundableitung}.
Then,
\[\liminf_{t \to \infty} \inf_{\tilde{\partial\rho^2_t}}\sup_{b\in\Sigma(\beta,\mathcal L)}\E_b\left[\left(\frac{\log t}{t}\right)^{-\frac{\beta}{2\beta +1}}\|\tilde{\partial\rho^2_t} - (\rho_b^2)'\|_\infty\right]>0\]
where the infimum is taken over all possible estimators $\tilde{\partial \rho^2_t}$ of $(\rho_b^2)'$.
\end{proposition}
\begin{proof}
The proof follows exactly the lines of the proof of lower bound for estimating $\rho_b'$, except from constructing the hypotheses in such a way that, for any $j\neq k$,
\[
\|(\rho_j^2)'-(\rho_k^2)'\|_\infty\ \geq\ C \left(\frac{\log t}{t}\right)^{\frac{\beta}{2\beta + 1}},
\]
$C$ some positive constant. 
This is achieved by choosing the kernel function $Q$ involved in the construction of the function $G_j$, $j\in J_t$, a little bit differently. 
Precisely, specify some function $Q\colon \R\to\R$ such that $Q\in C_c^\infty(\R)$, $\supp(Q) \subseteq [-\frac{1}{2},\frac{1}{2}]$, $Q\in\mathcal H(\beta+1,\frac{1}{2})$, $\int Q(x)\d x =0$, $Q(0)=0$ and $\left|Q'(0)\right|>0$.
For any $j\neq k, j,k \in J_t$, one then has
\begin{align*}
\|(\rho_j^2)'-(\rho_k^2)'\|_\infty &=\ 2\|(\rho'_0+G'_j)(\rho_0 + G_j) - (\rho'_0+G'_k)(\rho_0 + G_k)\|_\infty \\
&=\ 2\|\rho'_0G_j + G'_j\rho_0 + G'_jG_j  - \rho'_0G_k - G'_k\rho_0 - G'_kG_k\|_\infty\\
&\geq\ 2\left|\rho'_0(x_j)G_j(x_j) + G'_j(x_j)\rho_0(x_j) + G'_j(x_j)G_j(x_j)\right|\\
&\geq\ 2\mathcal L h_t^\beta \left|\rho'_0(x_j)h_tQ(0) + Q'(0)\rho_0(x_j) + \mathcal L h_t^{\beta + 1}Q(0)Q'(0)\right|\\  
&=\ 2\mathcal L h_t^\beta \left|Q'(0)\right|\rho_0(x_j)\ \geq\  2\mathcal L h_t^\beta \left|Q'(0)\right| \inf_{-A\leq x \leq A}\rho_0(x)\\
&=\ 2\mathcal L\left|Q'(0)\right| \inf_{-A\leq x \leq A}\rho_0(x) v^\beta\left(\frac{\log t}{t}\right)^{\frac{\beta}{2\beta + 1}}.
\end{align*}
Here we used the fact that $|x_j-x_k|\geq 2 h_t$ implies that 
\[x_j\notin \supp\left(Q'\left(\frac{\cdot-x_k}{h_t}\right)\right) \cup \supp\left(Q\left(\frac{\cdot-x_k}{h_t}\right)\right)\] because $\supp\left(Q'\left(\frac{\cdot-x_k}{h_t}\right)\right)\cup \supp\left(Q\left(\frac{\cdot-x_k}{h_t}\right)\right)\subseteq (x_k-h_t,x_k + h_t)$.
 The assertion then follows as in the previous proof (see Steps 1-3) from part (b) of Lemma \ref{lem:tsybakov}.
\end{proof}

In particular, Proposition \ref{thm:lower_bound_rho'rho} implies that condition \eqref{condition:lowdrift} from Lemma \ref{lemma:lowdrift} is fulfilled for $\psi_t=(\log t/t)^{\frac{\beta}{2\beta+1}}$, and \eqref{lowerbounddrift} follows.
\end{proof}

\medskip

\begin{proof}[Proof of Theorem \ref{theo:sim}]
The definition of $\hat h_t$ according to \eqref{est:band} implies that 
\begin{equation}\label{15.1}
\|\rho_{t,K}(\hat h_t)-\rho_{t,K}(h_{\min})\|_\infty\ \lesssim\ \frac{1}{\sqrt t\log t}.
\end{equation}
Furthermore, $h_{\min}$ satisfies the assumption of Proposition \ref{donskerdensity} such that assertion \eqref{sim:1} of the Theorem immediately follows.
It remains to verify \eqref{sim:21}.
First, we remark that it always holds $\hat h_t\geq h_{\min}$ such that $\hat h_t$ is well-defined.
For estimation of $\rho_b$ via $\rho_{t,K}(h_{\min})$, Proposition \ref{prop:con_kernel_density_estimator} yields 
\[
\sup_{b\in\Sigma(\beta,\mathcal L)} \left(\E_b\left[\left\|\rho_{t,K}(h_{\min})-\rho_b\right\|_\infty^p\right]\right)\p = O\left(t^{-1/2}(1+\sqrt{\log t} + \sqrt p + pt^{-1/2})\right). 
\] 
We can then deduce from \eqref{15.1} that 
\[
\sup_{b\in\Sigma(\beta,\mathcal L)} \left[\E_b\left(\left\|\rho_{t,K}(\hat h_t)-\rho_b\right\|_\infty\right)^p\right]\p =O\left(t^{-1/2}(1+\sqrt{\log t} + \sqrt p + pt^{-1/2})\right).
\]
Thus, $\rho_{t,K}(\hat h_t)$ satisfies assumption (E1) from Lemma \ref{lem:decomp}. 
We may now follow the proof of Theorem \ref{theo:est}.
In particular, it again suffices to investigate the estimator $\overline \rho_{t,K}(\hat h_t)$ in order to prove \eqref{sim:21} since conditions (E1) and (E2) from Lemma \ref{lem:decomp} are satisfied. 
Note that 
\[
\P_b\left(|\wideparen M-C\|\rho_b\|_\infty|>0.2C\|\rho_b\|_\infty\right)=\P_b\left(|\|\rho_{t,K}(h_{\min})\|_\infty-\|\rho_b\|_\infty|>0.2\|\rho_b\|_\infty\right) \leq\ \ t^{-1}.
\] 
follows exactly as in the proof of Theorem \ref{theo:est} since $\psi_{t,h_{\min}}(\log t)= o(1)$.
Additional arguments are required only for the investigation of Case 2 ($\hat h_t< \overline h_{\rho}$). 
As in the proof of Theorem \ref{theo:est}, it is shown that 
\begin{align*}
&\E_b\left[\|\overline\rho_{t,K}(\hat h_t)-\rho_b'/2\|_\infty\ \mathds{1}_{\{\hat h_t< \overline h_{\rho}\}\cap\{\wideparen M\geq 0.8M\}}\right]\\
&\qquad\leq\ \sum_{h\in\mathcal{H}\colon h<\overline h_{\rho}}\K\overline\sigma(h,t)\cdot\sqrt{\P_b\left(\{\hat h_t=h\}\cap\{0.8M\leq \wideparen M\}\right)}
+ B(\overline h_{\rho}).
\end{align*}
We bound the first term by 
$\sum_{h\in\HH\colon h<\overline h_{\rho}}\K\overline\sigma(h,t)\left(\sqrt{(\mathbf{I})}\ + \sqrt{(\mathbf{II})}\right)$,
with 
\begin{align*}
(\mathbf{I})&:=\ \sum_{g\in\HH\colon g\leq h} \P_b\left(\|\overline\rho_{t,K}(h^+)-\overline\rho_{t,K}(g)\|_\infty>\sqrt{0.8M}\overline\sigma(g,t)\right),\\
(\mathbf{II})&:=\ \P_b\left(\sqrt t\|\rho_{t,K}(h^+)-\rho_{t,K}(h_{\min})\|_\infty>\frac{\sqrt{h^+}(\log(1/h^+))^4}{\log t}\right),
\end{align*}
where $h^+:=\min\{g\in\mathcal{H}:\, g>h\}$.
$(\mathbf{I})$ is dealt with as in Theorem \ref{theo:est} (see \eqref{theo:einfachI}).

With regard to term $(\mathbf{II})$, we argue as before by means of Proposition \ref{theo:cath1}: 
Since $h_{\min}\leq h^+\leq \overline h_\rho$, it holds for any $\beta>0$ and an arbitrary positive constant $\co$, for some $t$ onwards,
\begin{align*}
\sqrt{h^+}\left( 1 + \log(1/\sqrt{h^+})+\log(t)\right)+\sqrt t\e^{-\co t}+\sqrt t(h^+)^{\beta+1}& =\ o(\lambda'),\\
\sqrt{h_{\min}}\left( 1 + \log(1/\sqrt{h_{\min}})+\log(t)\right)+\sqrt t\e^{-\co t}+\sqrt t(h_{\min})^{\beta+1}& =\ o(\lambda'),
\end{align*}
letting $\lambda':=\sqrt{h^+}(\log(1/h^+))^4/\log t$.
Thus, for any $t>1$ sufficiently large,
\begin{align*}
&\P_b\left(\sqrt t\|\rho_{t,K}(h^+)-\rho_{t,K}(h_{\min})\|_\infty>\lambda'\right)\\
&\quad\leq\ 
\P_b\left(\sqrt t\|\rho_{t,K}(h^+)-L_t^\bullet(X) t^{-1}\|_\infty>\frac{\lambda'}{2}\right) + \P_b\left(\sqrt t\|\rho_{t,K}(h_{\min})-L_t^\bullet(X) t^{-1}\|_\infty>\frac{\lambda'}{2}\right)\\
&\quad \leq\ 2\exp\left(-\frac{\Lambda_1(\log(1/h^+))^4}{\log t}\right)\ \leq\ 2\exp\left(-\widetilde{\Lambda}_1(\log t)^2\right),
\end{align*}
for some positive constant $\widetilde{\Lambda}_1$. Consequently,
\[
\sum_{h\in\HH\colon h<\overline h_{\rho}}\overline\sigma(h,t)\sqrt{(\mathbf{II})}
\ \lesssim\ \log t\cdot \overline{\sigma}(h_{\min},t)\sqrt{\exp(-\widetilde{\Lambda}_1(\log t)^2)}\ =\ o(\overline\sigma(\overline h_{\rho},t)).
\]
We can then proceed as in the proof of Theorem \ref{theo:est} to finish the proof.
\end{proof}

\section{Sketch of the derivation of basic exponential inequalities}\label{app:C}
In order to give some insight into the machinery of the proof of the basic concentration results, we sketch the main steps here in the situation of Proposition \ref{prop:csi}. 
Note that our previous paper \cp~deals with more general objects of the form
\[
\sqrt t \left(\frac{1}{t}\int_0^t f(X_s)\d X_s-\int(fb)\d\mu_b\right),\quad f\in\FF,\]
for some class $\FF$ of continuous functions, the supremum in the concentration results being taken over all functions $f\in\FF$.
The proof of Proposition \ref{prop:csi} relies on martingale approximation, i.e., the decomposition of $\overline\rho_{t,K}(h)-\E_b[\overline\rho_{t,K}(h)]$ into a negligible remainder term and a martingale part which is then dealt with by the generic chaining method.
Introduce the function 
\[g_y(u)\ :=\ \frac{1}{\rho_b(u)}\int_{\R}K\left(\frac{y-x}{h}\right)\rho_b'(x)\left(\mathds{1}\{u>x\}-F_b(u)\right)\d x,\quad u\in\R.\]
It\^o's formula applied to the function $\int_0^\bullet g_y(u)\d u$ and to the diffusion $X$ yields
\[
\int_{X_0}^{X_t}g_y(u)\d u\ =\ \int_0^tg_y(X_s)\d X_s+\frac{1}{2}\int_0^tg_y'(X_s)\d\langle X\rangle_s.
\]
Since
\[g_y'(u)\ =\ -2b(u)g_y(u)+2b(u)K\left(\frac{y-u}{h}\right)- 2\E_b\left[K\left(\frac{y-X_0}{h}\right)b(X_0)\right],
\]
we obtain the representation
\begin{align*}
\overline \rho_{t,K}(h)(y)-\E_b\left[\overline\rho_{t,K}(h)(y)\right]&=\ \frac{1}{t}\int_0^t K\left(\frac{y-X_s}{h}\right)\d X_s-\E_b\left[K\left(\frac{y-X_0}{h}\right)b(X_0)\right]\\
 &=\ \rd_t^y+\Ma_t^y,
\end{align*}
for
\begin{align*}
\rd_t^y&:=\ \int_{X_0}^{X_t}g_y(u)\d u,\\
\Ma_t^y&:=\ \int_0^tK\left(\frac{y-X_s}{h}\right)\d W_s-\int_0^tg_y(X_s)\d W_s\ =:\ \Ma_{1,t}^y+\Ma_{2,t}^y.
\end{align*}
It can be shown by direct calculations that
$ \left(\E_b\left[\sup_{y\in \R} |\rd_t^y|^p\right]\right)\p\leq \gamma(p)$
for some function $\gamma:\R_+\to\R_+$ (see Proposition 8 in \cp).

\paragraph{Step 2: Chaining applied to the martingale part}
Most effort is required for the analysis of the martingale part by the generic chaining method. 
In \cp, a more general setting is considered. 
Instead of the $\sup$-norm, the supremum over general function classes satisfying some mild entropy conditions is investigated. 
It is straightforward to translate our situation into this setting by defining the class $\mathcal F$ as in \eqref{functionclass}. 
The central generic chaining tool goes back to \citeauthor{tala14}.
We will use slightly modified versions of Theorems 3.2 and 3.5 in \cite{dirk15} which in particular yield upper bounds for \emph{all} moments. 
The latter turns out to be very useful for our statistical applications. 

\medskip

\begin{proposition}[cf.~Theorem 3.2 \& 3.5 in \cite{dirk15}]\label{thm:dirk}
Consider a real-valued process $(X_f)_{f\in\mathcal F}$, defined on a semi-metric space $(\FF,d)$.
\begin{itemize}
\item[$\operatorname{(a)}$] If there exists some $\alpha\in(0,\infty)$ such that
\begin{equation}\label{ineq:a}
\mathbb P\left(|X_f-X_g|\geq u d(f,g)\right)\ \leq\ 2\exp\left(-u^{\alpha}\right)\quad\forall  f,g\in\mathcal F,\ u\geq 1,
\end{equation}
then there exists some constant $C_\alpha>0$ (depending only on $\alpha$) such that, for any $1\leq p<\infty$,
\[
\left(\mathbb E\left[\sup_{f\in\mathcal F}|X_f|^p\right]\right)^{\frac{1}{p}}\ \leq\ C_\alpha \int_0^\infty \left(\log N(u,\mathcal F,d)\right)^{\frac{1}{\alpha}}\d u + 2\sup_{f\in\mathcal F}\left(\mathbb E\left[|X_f|^p\right]\right)^{\frac{1}{p}}.
\]
\item[$\operatorname{(b)}$] If there exist semi-metrics $d_1,d_2$ on $\FF$ such that
\[
\P\left(|X_f-X_g|\geq ud_1(f,g)+\sqrt u d_2(f,g)\right)\ \leq\ 2\e^{-u}\quad \forall f,g\in\FF,\ u\geq 1,
\]
then there exist positive constants $\tilde C_1,\tilde C_2$ such that, for any $1\leq p<\infty$,
\begin{align*}
\left(\mathbb E\left[\sup_{f\in\mathcal F}|X_f|^p\right]\right)^{\frac{1}{p}}
&\leq\ \tilde C_1 \int_0^\infty \log N(u,\FF,d_1)\d u \\
&\qquad\quad + \tilde C_2 \int_0^\infty \sqrt{\log N(u,\FF,d_2)}\d u+ 2\sup_{f\in\FF}\left(\mathbb E\left[|X_f|^p\right]\right)^{\frac{1}{p}}.
\end{align*}
\end{itemize}
\end{proposition}
Note that we upper bounded the so-called \emph{$\gamma$-functionals} appearing in Theorems 3.2 and 3.5 in \cite{dirk15} by entropy integrals. 
For the cases $\alpha=1,2$, part (a) of the Proposition deals with processes with subexponential or subgaussian increments, respectively. 
Part (b) covers the case of mixed-tail increments.

With regard to the conditions of the preceding proposition, we are led to exploring the tail behaviour of the increments of the martingale part which is done by finding suitable upper bounds for the moments. 
Indeed, taking into account assumption \eqref{kernel} on the kernel as well as the conditions on $b$, it can be shown that, for any $p\geq 1$, $y,z\in\R$,
\begin{eqnarray*}
\left(\E_b\left[\left(t^{-1/2}\left|\M_{2,t}^y\right|\right)^p\right]\right)\p&\leq& \left\|K\left(\frac{y-\cdot}{h}\right)\right\|_{\l^2(\lebesgue)}\cdot V(t,h)p^{\frac{3}{2}},\\
\left(\E_b\left[\left(t^{-1/2}\left|\M_{2,t}^y-\M_{2,t}^z\right|\right)^p\right]\right)\p&\leq&\left\|K\left(\frac{y-\cdot}{h}\right)-K\left(\frac{z-\cdot}{h}\right)\right\|_{\l^2(\lebesgue)}\cdot V(t,h)p^{\frac{3}{2}},
\end{eqnarray*}
for some function $V\colon\R^+\times\R^+\to\R^+$, so that we are led to part (a) of Proposition \ref{thm:dirk}. 
These moment bounds imply a corresponding exponential inequality of the form \eqref{ineq:a}. 
Of course, the dependence of $V$ on $h$ and $t$ is of immense importance for the result and the statistical application. 
Nevertheless, for ease of presentation, we skip details here and refer to Proposition 8 of \cp. 
The estimates therein are obtained by direct calculations, exploiting the conditions on the drift coefficient and the function class $\mathcal F$. 

The upper bound for $\Ma_{1,t}$ is based on another deep result on the $\sup$-norm of the centred local time which itself relies on the generic chaining. 
In order to motivate where this comes from, we sketch the main idea.
The Burkholder-Davis-Gundy inequality and the occupation times formula yield, for any $y\in\R$, and $p\geq 1$,
\begin{align*}
\left(\E_b\left[\left(\frac{1}{\sqrt t}|\Ma_{1,t}^y|\right)^p\right]\right)^{\frac{1}{p}}&=\  
\left(\E_b\left[\left|\frac{1}{\sqrt t}\int_0^t K\left(\frac{y-X_s}{h}\right)\d W_s\right|^p\right]\right)^{\frac{1}{p}}\\
&\leq\ 
C_p\left(\E_b\left[\left(\frac{1}{t}\int_0^tK^2\left(\frac{y-X_s}{h}\right)\d s\right)^{p/2}\right]\right)^{\frac{1}{p}}\\
&=\ 
 C_p\left(\E_b\left[\left(\frac{1}{t}\int_\R K^2\left(\frac{y-v}{h}\right)L_t^v(X)\d v\right)^{p/2}\right]\right)^{\frac{1}{p}}\\
&\leq\ 
 \frac{C_p}{\sqrt t}\left(\int_\R K^2\left(\frac{y-v}{h}\right)\d v\right)^{1/2}\left(\E_b\left[\left(\sup_{a\in\R}|L_t^a(X)|\right)^{p/2}\right]\right)^{\frac{1}{p}}\\
&\lesssim\ \frac{C_p}{\sqrt t}\left\|K\left(\frac{y-\cdot}{h}\right)\right\|_{L^2(\lebesgue)}\left\{\left(\E_b\left[\left(\sup_{a\in\R}|L_t^a(X)-\rho_b(a)|\right)^{p/2}\right]\right)^{\frac{1}{p}}+1\right\},
\end{align*}
$C_p$ denoting some positive constant which depends only on $p$. 
Exploiting the precise knowledge of the dependence of $C_p$ on $p$ and the result on the centred local time (Corollary 2 in \cp), one obtains the desired moment bounds and the corresponding exponential inequality:
\begin{align*}
\left[\E_b\left(t^{-1/2}\left|\M_{1,t}^y\right|\right)^p\right]\p&\leq\ \left\|K\left(\frac{y-\cdot}{h}\right)\right\|_{L^2(\lebesgue)}\cdot V_1(t,h)\left(\sqrt p + \frac{p}{t^{1/4}}\right), \text{ as well as}\\
\left[\E_b\left(t^{-1/2}\left|\M_{1,t}^y-\M_{1,t}^z\right|\right)^p\right]\p&\leq \ \left\|K\left(\frac{y-\cdot}{h}\right)-K\left(\frac{z-\cdot}{h}\right)\right\|_{L^2(\lebesgue)} V_1(t,h)\left(\sqrt p + \frac{p}{t^{1/4}}\right),
\end{align*}
for any $y,z\in\R$, $p\geq 1,$ and for some function $V_1\colon\R^+\times\R^+\to\R^+$. 
Hence, we would like to proceed with the chaining procedure for mixed-tail increments dealt with in part (b) of Proposition \ref{thm:dirk}.

\paragraph{Step 3: Localisation procedure}
Unfortunately, the metrics (basically the $L^2$-distance wrt the Lebesgue measure) for which the conditions of Proposition \ref{thm:dirk} are satisfied do not induce finite covering numbers. 
Therefore, we have to resort to a localisation procedure. 
The idea is to apply the chaining method to suitable subsets $\mathcal F_k\subseteq \FF$, $k\in\N_0$, with a decreasing probability for $\mathcal F_k$ to be relevant for large $k$. 
For suitably chosen compact intervals $I_k$, $k\in\N_0$, define
\[\mathcal F_k\ :=\ \left\{f\in\mathcal F:\, \supp(f)\subseteq I_k\right\}.\] 
The choice of $I_k$, $k\in\N_0$, implies that $\mathcal F=\bigcup_{k\in\N_0}\mathcal F_k$. 
One can then show that there exist sets $A^p_k$, $k\in\N_0$, for any $p\geq 1$ such that
\[A_f:=\{\text{there is } s\in [0,t]\text{ such that }X_s\in\supp(f)\}\subseteq A^p_k\] and 
$\P_b(A^p_k)\leq \exp(-kp)$, for all $f\in\mathcal F_k$, $k\in\N_0$, $p\geq 1$. 
This inequality is deduced from Lemma 1 in \cp~which states a maximal inequality for the diffusion process of the form 
\[\P_b\left(\max_{0\leq s\leq t}|X_s|>kp\Lambda t\right)\ \leq\ \e^{-kp},\quad k\in\N,\] for some constant $\Lambda>0$. 
The localisation procedure hinges on the compact support of the functions in $\mathcal F$. 
In the given set-up, the support of $K((x-\cdot)/h))$ for any $x\in\R$ is contained in $[x-h,x+h]$. 
Taking account of the maximal inequality, it thus becomes apparent that, for large $|x|$, the probability of $A_{K\left(\frac{x-\cdot}{h}\right)}$ is shrinking.
We can then deduce that
\begin{eqnarray*}
\left(\E_b\left[\sup_{y\in\R}\left|\M_t^y\right|^p\right]\right)\p 
&\leq &\sum_{k=0}^\infty \left(\E_b\left[\sup_{K\left(h^{-1}(y-\cdot)\right)\in\mathcal F_k}\left|\M_t^y\mathbbm{1}(A_f)\right|^p\right]\right)\p\\
&\leq& \sum_{k=0}^\infty \left(\E_b\left[\sup_{K\left(h^{-1}(y-\cdot)\right)\in\mathcal F_k}\left|\M_t^y\mathbbm{1}(A^p_k)\right|^{p}\right]\right)^{\frac{1}{p}}\\
&\leq& \sum_{k=0}^\infty \left(\E_b\left[\sup_{K\left(h^{-1}(y-\cdot)\right)\in\mathcal F_k}\left|\M_t^y\right|^{2p}\right]\right)^{\frac{1}{2p}}\e^{-\frac{k}{2}}.
\end{eqnarray*}
Proposition \ref{thm:dirk} can now be applied \emph{locally} to $\mathcal F_k$ and yields an upper bound on 
\[
	\left(\E_b\left[\sup_{K\left(h^{-1}(y-\cdot)\right)\in\mathcal F_k}\left|\M_t^y\right|^{2p}\right]\right)^{\frac{1}{2p}}.
\]
Taking into account the uniform upper bound on the remainder term, the first part of \eqref{con_stoch_int} follows from standard entropy bounds as in Lemma 22 in \cp.
Once these upper bounds for all $p$th moments are available, concentration inequalities as stated in the second part of \eqref{con_stoch_int} then immediately follow from the fact that, for any real valued random variable $Z$, satisfying, for any $p\geq 1$ and for some function $\phi\colon(0,\infty)\to(0,\infty)$, $\left(\E\left[|Z|^p\right]\right)\p\leq \phi(p)$, one has
\[\P\left(|Z|\geq \e \phi(u)\right)\ \leq\ \exp(-u),\quad u\geq 1.\]

\end{appendix}

\bibliography{exp_bib}   

\end{document}